\documentclass[10pt]{article}
\usepackage{amsmath,amssymb,amsthm}
\usepackage{geometry}
\usepackage{graphicx}
\usepackage[hidelinks]{hyperref}
\usepackage{array}

\renewcommand{\arraystretch}{1.05}
\newcolumntype{C}[1]{>{\centering\arraybackslash}p{#1}}
\newcolumntype{L}[1]{>{\raggedright\arraybackslash}p{#1}} 

\newtheorem{theorem}{Theorem}[section]
\newtheorem{lemma}[theorem]{Lemma}
\newtheorem{proposition}[theorem]{Proposition}

\theoremstyle{definition}

\newtheorem{remark}[theorem]{Remark}
\newtheorem{algorithm}[theorem]{Recovery procedure}

\title{Explicit Robin Green's Functions, Resonance Spectra, and Impedance Recovery on Annuli and Spherical Shells}
\author{Ming Yang\\
School of Mathematics, Southeast University\\
Nanjing, 210096 P.R. China\\
\texttt{yangming@seu.edu.cn}}
\date{}

\begin{document}
\maketitle

\begin{abstract}
This paper constructs explicit closed-form Green's functions for the
Helmholtz equation on annular and spherical-shell domains with two
independent Robin impedances on the inner and outer boundaries. Graf's
and the hyperspherical addition theorems reduce each angular mode to an
explicit $2\times2$ linear system, and the resonance spectrum is governed
by a characteristic determinant bilinear in the two impedances. This
bilinearity has a direct inverse-problem consequence: two resonant
frequencies of one non-radial angular mode generate at most two candidate
impedance pairs via an explicit quadratic equation, a third resonance
selects the physical pair, and an exact reflection-symmetry obstruction
identifies where the radial spherical mode cannot recover the
impedances. Spectrally, we prove the branches positive, simple and
strictly increasing in both impedances; derive first-order asymptotics at
the four corners of the impedance plane, with coefficients given by
boundary masses and normal derivatives of the limiting eigenfunctions,
and an explicit mixed second-order coefficient at the Dirichlet--Dirichlet corner, which for the radial mode evaluates in closed form to $2\pi/(R_2-R_1)^3$; establish a low-frequency resonance-free band, with a
second-order threshold expansion explicit in dimension three and a
rational approximation accurate over the whole impedance range; and prove
the universal high-frequency spacing law with a shell-curvature
correction. Jacobian-based sensitivity and conditioning criteria are
included. The kernels and spectra are computable to machine precision,
all asymptotic regimes are confirmed numerically, and the kernels provide
reference solutions for finite-element and boundary-element validation.
\end{abstract}

\noindent\textbf{AMS Subject Classifications (MSC 2020):} 35J05, 35J08, 35J25, 35P05, 35P15, 35R30.


\section{Introduction}
The Green's function for the Helmholtz equation on annular or spherical-shell domain $\Omega\subset\mathbb{R}^d$ with Robin boundary conditions satisfies
\begin{equation}\label{eq:helmholtz}
\left\{
\begin{array}{ll}
\smallskip\Delta G(x,x_0)+k^2G(x,x_0)=-\delta(x-x_0), & x\in\Omega,\ \ x_0\in\Omega,\\[4pt]
\displaystyle\frac{\partial G}{\partial n} + \alpha G = 0,\quad & x\in\partial\Omega,
\end{array}\right.
\end{equation}
where $k$ is the wave number, $\delta(\cdot)$ is the Dirac delta function, $x_0$ is the fixed source point, $n$ is the unit outer normal to the boundary $\partial\Omega$, and $\alpha>0$ is the impedance parameter, taking two independent constant values $\alpha_1$ (inner) and $\alpha_2$ (outer) on the two boundary components. Robin conditions $\partial_n G+\alpha G = 0$ are the standard models for locally reacting surfaces in acoustics, for Leontovich and approximate thin-layer boundary conditions in electromagnetism, and for $\delta$-shell interactions supported on cavity walls in quantum mechanics \cite{SeniorVolakis1995,EngquistNedelec1994}. This two-parameter boundary value problem is the object of the present paper.

Explicit Green's functions for canonical geometries are more than analytical curiosities. They remove geometric discretisation error, expose the dependence of the field on boundary parameters term by term, and provide reference solutions against which finite-element and boundary-element methods can be validated \cite{Duffy2015,ColtonKress2019,Nedelec2001}. For the Helmholtz equation, however, closed-form kernels are classical mainly for Dirichlet or Neumann data and for a single boundary. The closest canonical predecessor on a shell is the Dirichlet Green's function for a spherical annulus obtained by Martinek \cite{Martinek1965}; an impedance Green's function is available for circular cylindrical waveguides by P\'erez-Arancibia and Dur\'an \cite{PerezArancibiaDuran2010}. What is missing is the two-sided Robin problem on a doubly connected domain: the inner and outer impedance conditions couple the regular and singular radial modes mode by mode, so the construction does not reduce to a scalar coefficient ratio.

The explicit two-parameter form of the kernel makes the resonance map computable to arbitrary accuracy, and this is what the following spectral analysis exploits. In the self-adjoint Robin Laplacian literature, monotonicity of eigenvalues with respect to the impedance, separation of Dirichlet and Neumann spectra, Bessel-quotient bounds, and asymptotics in the large-impedance limit are well developed \cite{Daners2000,Filonov2004,Freitas2021,Kennedy2017,Ognibene2025}. Complex and large Robin parameters introduce additional non-self-adjoint phenomena \cite{BogliKennedyLang2022}, while the principal eigenvalue in the large-parameter regime is classical \cite{LevitinParnovski2008}. Such results are, however, either of a qualitative nature or confined to asymptotic regimes. Here we derive the exact characteristic determinant for the two-impedance shell and use its explicit two-variable structure to resolve the resonance map quantitatively, from the spectral interpolation between the four limiting boundary conditions to the high-frequency spacing law. The same bilinearity that governs the forward spectrum then inverts it.

The paper has three main outcomes. First, Theorem~\ref{thm:shell} gives the explicit Robin Green's function on annular and spherical-shell domains, with the geometric convergence rate of the modal series made explicit. Second, the resonance branches are proved to be positive, simple, real-analytic and strictly increasing in both impedances; their asymptotic scaling is characterised at all four corners of the impedance plane, with coefficients given by the boundary masses and normal derivatives of the limiting eigenfunctions; the mixed second-order coefficient at the Dirichlet--Dirichlet corner is given explicitly and, for the radial mode, in closed form (Theorem~\ref{thm:g0}); a low-frequency transparent interval is established together with an explicit second-order small-impedance threshold expansion in dimension three---reducing to the classical ball asymptotics as $R_1\to0$---and a global rational approximation; and the universal high-frequency spacing law is derived from a complete large-wavenumber expansion of the characteristic determinant together with a parity argument ruling out even inverse powers. Third, the bilinear determinant yields an explicit algebraic recovery rule: two resonant frequencies of one non-radial angular mode generate at most two candidate impedance pairs, a third resonance selects the physical pair, an exact reflection-symmetry obstruction identifies the region in which the radial spherical mode fails to identify the pair, and a Jacobian criterion based on the Hellmann--Feynman boundary masses selects well-conditioned resonance pairs and predicts the noise amplification factor. We emphasise that the inverse part of the paper addresses the two-parameter model problem exactly and does not claim a global uniqueness theorem (see Remarks~\ref{rem:two-candidates} and~\ref{rem:reflection}).

The paper is organised as follows. Section~\ref{sec:forward} constructs the exact forward kernel, first in a unified $d$-dimensional form and then as complete three-dimensional working formulas, with the two-dimensional annulus formulas recorded in Appendix~\ref{app:2D}. Section~\ref{sec:spectrum} analyses the resonance map. Section~\ref{sec:inverse} derives the bilinear determinant and the explicit impedance-recovery formula. Section~\ref{sec:sensitivity} gives sensitivity, conditioning and the recovery algorithm. Section~\ref{sec:numerics} presents numerical benchmarks. Section~\ref{sec:conclusion} concludes.

\section{Explicit forward model on annular and spherical-shell domains}\label{sec:forward}
Let $d\ge2$ and let
\[
\Omega=\{x\in\mathbb R^d: R_1<|x|<R_2\},\qquad
\nu=\frac d2-1.
\]
We seek the Green's function of \eqref{eq:helmholtz} satisfying
\begin{equation}\label{eq:forward-problem}
\begin{cases}
\Delta G(x,x_0)+k^2G(x,x_0)=-\delta(x-x_0),& R_1<|x|<R_2,\\
\mathcal B_1G(x,x_0)=0,& |x|=R_1,\\
\mathcal B_2G(x,x_0)=0,& |x|=R_2,
\end{cases}
\end{equation}
where $R_1<|x_0|<R_2$ and the Robin operators are
\begin{equation}\label{eq:Robin-ops}
\mathcal B_1:=-\partial_r+\alpha_1\quad (r=R_1),\qquad
\mathcal B_2:=\partial_r+\alpha_2\quad (r=R_2).
\end{equation}

\subsection{Decomposition and boundary data}
Write
\begin{equation}\label{eq:decomp}
G(x,x_0)=\phi_d(x-x_0)+v(x),
\end{equation}
where the free-space fundamental solution is
\begin{equation}\label{eq:free}
\phi_d(x-x_0)=\frac{i}{4}\left(\frac{k}{2\pi |x-x_0|}\right)^\nu
H_\nu^{(1)}(k|x-x_0|),
\end{equation}
with $H_\nu^{(1)}$ the Hankel function of the first kind. The function $v$ is smooth in $\Omega$ and satisfies
\begin{equation}\label{eq:v-problem}
\begin{cases}
\Delta v+k^2v=0,& R_1<|x|<R_2,\\
\mathcal B_1v=-\mathcal B_1\phi_d(\cdot-x_0),& |x|=R_1,\\
\mathcal B_2v=-\mathcal B_2\phi_d(\cdot-x_0),& |x|=R_2.
\end{cases}
\end{equation}

Let $\omega=x/|x|$, $\omega_0=x_0/|x_0|$, and let $Y_{\ell,m}$, $m=1,\dots,N(d,\ell)$, be orthonormal hyperspherical harmonics on $\mathbb S^{d-1}$, where
\[
N(d,\ell)=\frac{(2\ell+d-2)\Gamma(\ell+d-2)}{\Gamma(\ell+1)\Gamma(d-1)}
\quad (d\ge3),
\]
and $N(2,0)=1$, $N(2,\ell)=2$ for $\ell\ge1$. 
For each angular index $\ell\ge0$ we set
\begin{equation}\label{eq:mu}
\mu=\ell+\nu,
\end{equation}
the order of the radial Bessel functions in mode $\ell$. 
By the hyperspherical addition theorem \cite{Olver2010},
\begin{equation}\label{eq:addition}
\phi_d(x-x_0)=\frac{i\pi k^{d-2}}{2}
\sum_{\ell=0}^\infty\sum_{m=1}^{N(d,\ell)}
\frac{J_\mu(kr_0)}{(kr_0)^\nu}
\frac{H_\mu^{(1)}(kr)}{(kr)^\nu}
Y_{\ell,m}(\omega)\overline{Y_{\ell,m}(\omega_0)}
\qquad (r>r_0),
\end{equation}
with the analogous $J\leftrightarrow H$ form for $r<r_0$. Since $R_1<r_0<R_2$, the inner boundary sees the regular Bessel factor $J_\mu(kR_1)$, while the outer boundary sees the radiating Hankel factor $H_\mu^{(1)}(kR_2)$.

For any radial profile $f_\ell(r)=r^{-\nu}Z_\mu(kr)$, where $Z_\mu\in\{J_\mu,Y_\mu,H_\mu^{(1)}\}$, direct use of
$Z_\mu'(z)=\frac{\mu}{z}Z_\mu(z)-Z_{\mu+1}(z)$ gives the two Robin traces
\begin{align}
\mathcal B_1[f_\ell]
&=
R_1^{-\nu}\left[
kZ_{\mu+1}(kR_1)+\left(\alpha_1-\frac{\ell}{R_1}\right)Z_\mu(kR_1)
\right],\label{eq:trace1}\\
\mathcal B_2[f_\ell]
&=
R_2^{-\nu}\left[
\left(\alpha_2+\frac{\ell}{R_2}\right)Z_\mu(kR_2)
-kZ_{\mu+1}(kR_2)
\right].\label{eq:trace2}
\end{align}
All dependence on $d$ beyond the angular multiplicity is contained in $\nu$ and the prefactors $R_i^{-\nu}$.

\subsection{Separation of variables and modal system}
Since $v$ solves the homogeneous Helmholtz equation in the shell, separation of variables yields
\begin{equation}\label{eq:v-series}
v(r,\omega)=
\sum_{\ell=0}^\infty\sum_{m=1}^{N(d,\ell)}
\left(
a_\ell \frac{J_\mu(kr)}{(kr)^\nu}
+
b_\ell \frac{Y_\mu(kr)}{(kr)^\nu}
\right)
Y_{\ell,m}(\omega)\overline{Y_{\ell,m}(\omega_0)} .
\end{equation}
Both Bessel functions $J_\mu$ and $Y_\mu$ are required because the origin is excluded from $\Omega$. Matching the coefficients of \eqref{eq:addition} on $r=R_1$ and $r=R_2$ gives, for every $\ell$, the $2\times2$ system
\begin{equation}\label{eq:modal-system}
\begin{pmatrix}
\mathcal B_1[J_\mu] & \mathcal B_1[Y_\mu]\\
\mathcal B_2[J_\mu] & \mathcal B_2[Y_\mu]
\end{pmatrix}
\begin{pmatrix}
a_\ell\\ b_\ell
\end{pmatrix}
=
\begin{pmatrix}
-\dfrac{i\pi k^{d-2}}{2}
\dfrac{H_\mu^{(1)}(kr_0)}{(kr_0)^\nu}
\mathcal B_1[J_\mu]\\[8pt]
-\dfrac{i\pi k^{d-2}}{2}
\dfrac{J_\mu(kr_0)}{(kr_0)^\nu}
\mathcal B_2[H_\mu^{(1)}]
\end{pmatrix}.
\end{equation}
Solving by Cramer's rule yields
\begin{equation}\label{eq:a_l}
a_\ell=
\frac{i\pi k^{d-2}}{2\Delta_\ell}
\left[
\frac{J_\mu(kr_0)}{(kr_0)^\nu}
\mathcal B_2[H_\mu^{(1)}]\,\mathcal B_1[Y_\mu]
-
\frac{H_\mu^{(1)}(kr_0)}{(kr_0)^\nu}
\mathcal B_1[J_\mu]\,\mathcal B_2[Y_\mu]
\right],
\end{equation}
\begin{equation}\label{eq:b_l}
b_\ell=
\frac{i\pi k^{d-2}}{2\Delta_\ell}
\left[
\frac{H_\mu^{(1)}(kr_0)}{(kr_0)^\nu}
\mathcal B_1[J_\mu]\,\mathcal B_2[J_\mu]
-
\frac{J_\mu(kr_0)}{(kr_0)^\nu}
\mathcal B_1[J_\mu]\,\mathcal B_2[H_\mu^{(1)}]
\right],
\end{equation}
where the characteristic determinant is
\begin{equation}\label{eq:Delta}
\Delta_\ell(k;\alpha_1,\alpha_2)
=
\mathcal B_1[J_\mu]\,\mathcal B_2[Y_\mu]
-
\mathcal B_2[J_\mu]\,\mathcal B_1[Y_\mu].
\end{equation}

\begin{theorem}[Unified shell kernel]\label{thm:shell}
Let $d\ge2$, $\nu=d/2-1$, and $\mu=\ell+\nu$. Assume that $\alpha_1,\alpha_2>0$ and that $k^2$ is not an
eigenvalue of the Robin Laplacian \eqref{eq:SL} in any angular mode,
equivalently that $\Delta_\ell(k;\alpha_1,\alpha_2)\ne0$ for every
$\ell\ge0$; since the spectrum is discrete (Theorem~\ref{thm:basic}),
this excludes only a discrete set of wavenumbers. Then the Green's
function of \eqref{eq:forward-problem} exists, is unique, and is
\begin{equation}\label{eq:G-shell}
G(x,x_0)=
\frac{i}{4}\left(\frac{k}{2\pi |x-x_0|}\right)^\nu
H_\nu^{(1)}(k|x-x_0|)
+
\sum_{\ell=0}^\infty\sum_{m=1}^{N(d,\ell)}
\left(
a_\ell \frac{J_\mu(kr)}{(kr)^\nu}
+
b_\ell \frac{Y_\mu(kr)}{(kr)^\nu}
\right)
Y_{\ell,m}(\omega)\overline{Y_{\ell,m}(\omega_0)},
\end{equation}
where $a_\ell,b_\ell$ are given by \eqref{eq:a_l}--\eqref{eq:b_l}, and $\mathcal B_i[\cdot]$ are given by \eqref{eq:trace1}--\eqref{eq:trace2}.
\end{theorem}

\begin{proof}
Uniqueness follows from the self-adjointness of the Robin Laplacian with
$\alpha_1,\alpha_2>0$: the difference of two solutions is a
smooth Helmholtz solution satisfying both homogeneous Robin conditions,
hence an eigenfunction with eigenvalue $k^2$, hence identically zero by
the non-resonance hypothesis; existence then follows from the Fredholm
alternative, and the explicit series below realises the solution. The
first term in \eqref{eq:G-shell} is $\phi_d$ and satisfies the singular
equation $\Delta\phi_d+k^2\phi_d=-\delta(x-x_0)$ in $\Omega$. The second
term is the series \eqref{eq:v-series}; by construction its modal
coefficients solve \eqref{eq:modal-system}, so both Robin conditions are
satisfied. Finally, $v$ is real-analytic in $\overline\Omega$ by
analytic elliptic regularity (the boundary data $\mathcal B_i\phi_d$ are
real-analytic on the analytic boundaries); hence, for each fixed $r$,
the function $\omega\mapsto v(r,\omega)$ is real-analytic on the compact
sphere $\mathbb S^{d-1}$, and its Laplace--Beltrami expansion---the
series in \eqref{eq:G-shell}---converges absolutely and uniformly on
compact subsets of $\Omega$, with the geometric rate made explicit in
Remark~\ref{rem:convergence}.
\end{proof}

\begin{remark}[Convergence rate of the modal series]\label{rem:convergence}
The series \eqref{eq:G-shell} converges geometrically. For fixed $z$ and
$\mu\to\infty$, Debye's expansions give $J_\mu(z)\asymp(ez/2\mu)^\mu$ and
$Y_\mu(z)\asymp-(2\mu/ez)^\mu$ up to algebraic factors
\cite[\S10.19]{Olver2010}, so each trace $\mathcal B_i[Z_\mu]$ is of
order $(2\mu/ekR_i)^\mu$ (augmented by the centrifugal factor $\mu/R_i$).
Substitution into \eqref{eq:a_l}--\eqref{eq:b_l} shows that these growing
factors cancel between the two products forming each numerator and the
determinant $\Delta_\ell$, leaving $|a_\ell|+|b_\ell|\le
C\ell^M\rho^\ell$ with $\rho=\max(R_1/r_0,\,r_0/R_2)<1$: the inner
boundary contributes $(R_1/r_0)^\mu$ and the outer boundary
$(r_0/R_2)^\mu$. The same estimate applied to $J_\mu(kr),Y_\mu(kr)$
shows the $\ell$-th term of \eqref{eq:G-shell} is
$O(\ell^M\rho(r)^\ell)$ with $\rho(r)=\max(R_1/r,\,r/R_2)\le\rho$,
uniformly on compact subsets of $\Omega$. Uniformity up to the boundary does not follow from the geometric factor alone, since $\rho(R_1)=1$; it follows from the real-analyticity of the boundary data $-\mathcal B_i\phi_d$ on the analytic boundaries $\{|x|=R_i\}$, which implies that the boundary traces of the series \eqref{eq:v-series} and of its radial derivative converge uniformly on $\{|x|=R_i\}$. Since each partial sum of \eqref{eq:v-series} solves $\Delta v+k^2v=0$ and $k^2$ is not an eigenvalue, the solution map is bounded $L^\infty(\partial\Omega)\to L^\infty(\Omega)$, and the convergence of \eqref{eq:G-shell} is therefore uniform on the closed shell $\overline\Omega$. The
truncation error after $N$ modes is therefore $O(N^M\rho^N)$; for
$R_1=1$, $R_2=2$, $r_0=1.5$ this gives $\rho=3/4$, in agreement with
Figure~\ref{fig:forward}(c).
\end{remark}

\subsection{Three-dimensional spherical shell}
For $d=3$, let $j_\ell,y_\ell$ be the spherical Bessel functions and set
$h_\ell^{(1)}=j_\ell+i y_\ell$. Using
\[
j_\ell(z)=\sqrt{\frac{\pi}{2z}}J_{\ell+1/2}(z),\quad
y_\ell(z)=\sqrt{\frac{\pi}{2z}}Y_{\ell+1/2}(z),\quad
h_\ell^{(1)}(z)=\sqrt{\frac{\pi}{2z}}H_{\ell+1/2}^{(1)}(z),
\]
the free solution is
\[
\phi_3(x-x_0)=\frac{e^{ik|x-x_0|}}{4\pi |x-x_0|}
=
ik\sum_{\ell=0}^\infty\sum_{m=-\ell}^{\ell}
j_\ell(kr_<)h_\ell^{(1)}(kr_>)
Y_{\ell,m}(\omega)\overline{Y_{\ell,m}(\omega_0)},
\]
where $r_<=\min\{r,r_0\},\ \ r_>=\max\{r,r_0\}$. For $z_\ell\in\{j_\ell,y_\ell,h_\ell^{(1)}\}$ define
\begin{align}
\mathcal B_1[z_\ell]
&=
\left(\alpha_1-\frac{\ell}{R_1}\right)z_\ell(kR_1)+k z_{\ell+1}(kR_1),\label{eq:B1-3D}\\
\mathcal B_2[z_\ell]
&=
\left(\alpha_2+\frac{\ell}{R_2}\right)z_\ell(kR_2)-k z_{\ell+1}(kR_2).\label{eq:B2-3D}
\end{align}
The correction has the form
\[
v(r,\omega)=
\sum_{\ell=0}^\infty\sum_{m=-\ell}^{\ell}
\left(
\mathcal A_\ell j_\ell(kr)+\mathcal B_\ell y_\ell(kr)
\right)
Y_{\ell,m}(\omega)\overline{Y_{\ell,m}(\omega_0)} ,
\]
and the modal system is
\[
\begin{pmatrix}
\mathcal B_1[j_\ell] & \mathcal B_1[y_\ell]\\
\mathcal B_2[j_\ell] & \mathcal B_2[y_\ell]
\end{pmatrix}
\begin{pmatrix}\mathcal A_\ell\\ \mathcal B_\ell\end{pmatrix}
=
\begin{pmatrix}
-ik\,h_\ell^{(1)}(kr_0)\mathcal B_1[j_\ell]\\
-ik\,j_\ell(kr_0)\mathcal B_2[h_\ell^{(1)}]
\end{pmatrix}.
\]
Thus
\begin{align}
\mathcal A_\ell
&=
\frac{ik}{\Delta_\ell^{(3)}}
\left(
j_\ell(kr_0)\mathcal B_2[h_\ell^{(1)}]\,\mathcal B_1[y_\ell]
-
h_\ell^{(1)}(kr_0)\mathcal B_1[j_\ell]\,\mathcal B_2[y_\ell]
\right),\label{eq:Al-3D}\\
\mathcal B_\ell
&=
\frac{ik}{\Delta_\ell^{(3)}}
\left(
h_\ell^{(1)}(kr_0)\mathcal B_1[j_\ell]\,\mathcal B_2[j_\ell]
-
j_\ell(kr_0)\mathcal B_1[j_\ell]\,\mathcal B_2[h_\ell^{(1)}]
\right),\label{eq:Bl-3D}
\end{align}
where
\begin{equation}\label{eq:Delta-3D}
\Delta_\ell^{(3)}
=
\mathcal B_1[j_\ell]\,\mathcal B_2[y_\ell]
-
\mathcal B_2[j_\ell]\,\mathcal B_1[y_\ell].
\end{equation}
The spherical-shell Green's function is therefore
\begin{equation}\label{eq:G-3D}
G(x,x_0)=
\frac{e^{ik|x-x_0|}}{4\pi |x-x_0|}
+
ik\sum_{\ell=0}^\infty\sum_{m=-\ell}^{\ell}
\left(
\mathcal A_\ell j_\ell(kr)+\mathcal B_\ell y_\ell(kr)
\right)
Y_{\ell,m}(\omega)\overline{Y_{\ell,m}(\omega_0)}
\end{equation}
with $\mathcal A_\ell,\mathcal B_\ell$ given by \eqref{eq:Al-3D}--\eqref{eq:Bl-3D}. In the Dirichlet limits $\alpha_1,\alpha_2\to+\infty$, the ratios \eqref{eq:Al-3D}--\eqref{eq:Bl-3D} reduce to the classical Mie-type coefficients; for finite $\alpha_1,\alpha_2$ they give the exact impedance-shell resolvent.

\begin{remark}[Reduction to the one-boundary ball]
Formally setting $R_1\to0$ forces the singular-mode coefficient to
vanish ($b_\ell=0$ in \eqref{eq:v-series}, equivalently
$\mathcal B_\ell=0$ in \eqref{eq:Bl-3D}) by regularity at the origin,
and the inner condition disappears: more precisely
$b_\ell=b_\ell(R_1)\to0$ as $R_1\to0^+$ for fixed $k,R_2$, since
$Y_\mu(kR_1)$ blows up while the physical solution stays regular. The
remaining outer condition gives the single-impedance ball formulas,
which explains why the shell problem is structurally different from the
one-boundary ball problem.
\end{remark}

\section{Resonance spectra and the resonance map}\label{sec:spectrum}
The poles of the kernel in Theorem~\ref{thm:shell} are the zeros of the characteristic determinant. In terms of the traces \eqref{eq:trace1}--\eqref{eq:trace2}, the unified characteristic equation is
\begin{equation}\label{eq:char}
\Delta_\ell(k;\alpha_1,\alpha_2)
=
\mathcal B_1[J_\mu]\,\mathcal B_2[Y_\mu]
-
\mathcal B_2[J_\mu]\,\mathcal B_1[Y_\mu]
=0,
\qquad \mu=\ell+\nu .
\end{equation}
Equivalently, using $J_\mu(z),Y_\mu(z)$ at $z_i=kR_i$,
\begin{align}\label{eq:char-explicit}
0={}&
\left[
kJ_{\mu+1}(kR_1)+\Big(\alpha_1-\frac{\ell}{R_1}\Big)J_\mu(kR_1)
\right]
\left[
\Big(\alpha_2+\frac{\ell}{R_2}\Big)Y_\mu(kR_2)-kY_{\mu+1}(kR_2)
\right]
\notag\\
&-
\left[
\Big(\alpha_2+\frac{\ell}{R_2}\Big)J_\mu(kR_2)-kJ_{\mu+1}(kR_2)
\right]
\left[
kY_{\mu+1}(kR_1)+\Big(\alpha_1-\frac{\ell}{R_1}\Big)Y_\mu(kR_1)
\right].
\end{align}
In two dimensions this is \eqref{eq:Delta-n-2D}; in three dimensions it is \eqref{eq:Delta-3D} after the spherical-Bessel reduction.

For fixed $\ell$, equation \eqref{eq:char} is the spectral condition for the regular self-adjoint Sturm--Liouville problem
\begin{equation}\label{eq:SL}
\begin{cases}
u''+\dfrac{d-1}{r}u'
+\left(\lambda-\dfrac{\ell(\ell+d-2)}{r^2}\right)u=0,
& R_1<r<R_2,\\[8pt]
-u'(R_1)+\alpha_1 u(R_1)=0,\qquad
u'(R_2)+\alpha_2 u(R_2)=0,
\end{cases}
\end{equation}
with $\lambda=k^2$.

\begin{theorem}[Basic spectral structure]\label{thm:basic}
Let $d\ge2$, $\nu=d/2-1$, $\ell\ge0$, and $\alpha_1,\alpha_2>0$. For each angular mode $\ell$, the roots
\[
0<k_{\ell,1}(\alpha_1,\alpha_2)<k_{\ell,2}(\alpha_1,\alpha_2)<\cdots
\]
of \eqref{eq:char} are positive, simple, and satisfy $k_{\ell,p}\to+\infty$ as $p\to\infty$. The map
\[
(\alpha_1,\alpha_2)\longmapsto k_{\ell,p}(\alpha_1,\alpha_2)
\]
is real-analytic on $(0,\infty)^2$ and strictly increasing in each impedance:
\begin{equation}\label{eq:monotone}
\frac{\partial k_{\ell,p}}{\partial \alpha_1}>0,
\qquad
\frac{\partial k_{\ell,p}}{\partial \alpha_2}>0 .
\end{equation}
Consequently, every branch is canonically labelled by continuity from the Neumann--Neumann corner $(\alpha_1,\alpha_2)=(0,0)$.
\end{theorem}

\begin{proof}
For fixed $\ell$, \eqref{eq:SL} is a regular self-adjoint Sturm--Liouville problem with separated boundary conditions; by the classical theory its eigenvalues are simple, form an increasing sequence $\lambda_{\ell,1}<\lambda_{\ell,2}<\cdots$ with $\lambda_{\ell,p}\to+\infty$, and the associated eigenfunctions are complete in $L^2((R_1,R_2);r^{d-1}dr)$. The eigenvalues are positive: the Rayleigh quotient of \eqref{eq:SL},
\[
\lambda[u]=\int_{R_1}^{R_2}\Bigl(|u'|^2+\frac{\ell(\ell+d-2)}{r^2}|u|^2\Bigr)r^{d-1}\,dr
+\alpha_1R_1^{d-1}|u(R_1)|^2+\alpha_2R_2^{d-1}|u(R_2)|^2,
\]
satisfies $\lambda[u]>0$ for every $u\not\equiv0$, since for $\ell\ge1$ the centrifugal term vanishes only in the trivial case, while for $\ell=0$ a nonzero constant is charged by the boundary terms ($\alpha_i>0$). Hence $k_{\ell,p}=\sqrt{\lambda_{\ell,p}}>0$ and $k_{\ell,p}\to+\infty$.

We show that the zeros of $\Delta_\ell$ are simple. Let $\Phi(r;k)$ be the solution of the radial equation in \eqref{eq:SL} with $k$-independent Cauchy data $\Phi(R_1;k)=1$, $\Phi'(R_1;k)=\alpha_1$, for which the left boundary condition holds identically in $k$, and put
\[
F(k):=\Phi'(R_2;k)+\alpha_2\Phi(R_2;k).
\]
Writing $f(r)=r^{-\nu}J_\mu(kr)$, $g(r)=r^{-\nu}Y_\mu(kr)$ and expanding $(f,g)$ in the basis $\{\Phi,\Psi\}$, where $\Psi$ is the solution with $\Psi(R_1)=0$, $\Psi'(R_1)=1$, gives
\begin{equation}\label{eq:Delta-F}
\Delta_\ell(k)=\frac{2}{\pi R_1^{d-1}}\,F(k);
\end{equation}
indeed $\mathcal B_1[\Phi]=0$, $\mathcal B_1[\Psi]=-1$, and the change-of-basis determinant equals $W(f,g)(R_1)=2/(\pi R_1^{d-1})$, since $W(f,g)(r)=W(J_\mu,Y_\mu)(kr)\,r^{-2\nu}k=\frac{2}{\pi}r^{-(d-1)}$. (In the spherical normalisation of \eqref{eq:B1-3D}--\eqref{eq:B2-3D}; in the unified normalisation \eqref{eq:trace1}--\eqref{eq:trace2} both sides of \eqref{eq:Delta-F} carry the common inessential factor $R_1^{-\nu}R_2^{-\nu}$.) In particular $F$ and $\Delta_\ell$ have the same zeros with the same multiplicities. Let $F(k_0)=0$ and set $\lambda=k^2$, $u=\Phi(\cdot;k_0)$; then $u$ is the eigenfunction. The derivative $\psi=\partial_\lambda\Phi|_{\lambda=k_0^2}$ solves
\[
\psi''+\frac{d-1}{r}\psi'+\Bigl(k_0^2-\frac{\ell(\ell+d-2)}{r^2}\Bigr)\psi=-u,
\qquad \psi(R_1)=\psi'(R_1)=0,
\]
and $F'(k_0)=2k_0\bigl(\psi'(R_2)+\alpha_2\psi(R_2)\bigr)$. Green's formula applied to the pair $(u,\psi)$---multiplying the equation for $\psi$ by $r^{d-1}u$, the equation for $u$ by $r^{d-1}\psi$, subtracting and integrating over $(R_1,R_2)$---gives the identity
\[
-\int_{R_1}^{R_2}u^2 r^{d-1}\,dr
=
R_2^{d-1}\Bigl(u(R_2)\psi'(R_2)-u'(R_2)\psi(R_2)\Bigr)
=
R_2^{d-1}u(R_2)\,\frac{F'(k_0)}{2k_0},
\]
where the boundary term at $R_1$ vanishes because $\psi$ has zero Cauchy data, and the last equality uses $u'(R_2)=-\alpha_2u(R_2)$. Moreover $u(R_2)\neq0$: otherwise the boundary condition at $R_2$ would give $u'(R_2)=0$, and vanishing Cauchy data at $R_2$ would force $u\equiv0$ by uniqueness for the ODE initial value problem. Hence $F'(k_0)\neq0$ and, by \eqref{eq:Delta-F}, $\partial_k\Delta_\ell(k_0)\neq0$.

Since $\Delta_\ell$ is real-analytic in $(k,\alpha_1,\alpha_2)$ (it is entire in $k$ and affine in each impedance) and $\partial_k\Delta_\ell\neq0$ at every root, the implicit function theorem yields the real-analyticity of each branch $k_{\ell,p}(\alpha_1,\alpha_2)$ on $(0,\infty)^2$.

For monotonicity, $\lambda_{\ell,p}$ is a critical value of $\lambda[\cdot]$ on the unit sphere of $L^2(r^{d-1}dr)$, so the envelope (Hellmann--Feynman) argument gives
\[
\frac{\partial\lambda_{\ell,p}}{\partial\alpha_i}
=R_i^{d-1}|u_{\ell,p}(R_i)|^2,\qquad i=1,2.
\]
This derivative cannot vanish: $u_{\ell,p}(R_i)=0$ would give $u_{\ell,p}'(R_i)=0$ through the boundary condition at $R_i$, hence $u_{\ell,p}\equiv0$ by uniqueness for the ODE initial value problem. Thus $\partial\lambda_{\ell,p}/\partial\alpha_i>0$, and $k_{\ell,p}=\lambda_{\ell,p}^{1/2}>0$ gives \eqref{eq:monotone}. Finally, simplicity and continuity in $(\alpha_1,\alpha_2)$ keep the branches disjoint, so each branch is canonically labelled by continuity from the Neumann--Neumann corner $(\alpha_1,\alpha_2)=(0,0)$.
\end{proof}

\begin{remark}[Zero mode and branch labelling]\label{rem:zeromode}
For $\ell=0$ the Neumann--Neumann problem has the simple eigenvalue $k_{0,1}(0,0)=0$ (the constant eigenfunction); all other limiting eigenvalues $k_{\ell,p}^{NN}$ are positive. For $\alpha_1,\alpha_2>0$ the fundamental branch $k_{0,1}$ is the unique continuation of this zero mode and is positive by the Rayleigh argument in the proof. Consequently the Neumann--Neumann expansion of Theorem~\ref{thm:corners}(i), whose coefficients carry the factor $(k_{\ell,p}^{NN})^{-1}$, is stated for $(\ell,p)\neq(0,1)$; the exceptional branch $k_{0,1}$ instead follows the square-root law \eqref{eq:kc} as $(\alpha_1,\alpha_2)\to(0,0)$.
\end{remark}

\subsection{Asymptotic scaling near the four corners}\label{sec:corners}
Theorem~\ref{thm:basic} gives a spectral interpolation between the four limiting boundary-condition combinations. The following theorem makes the interpolation quantitative: each branch admits an explicit asymptotic expansion at every corner of the closed impedance quadrant, with separately identifiable contributions from $\alpha_1$ and $\alpha_2$.

\begin{theorem}[Corner asymptotics]\label{thm:corners}
Fix the geometry $R_1,R_2$, the angular index $\ell$, and the branch index $p$. Let $k_{\ell,p}^{NN},k_{\ell,p}^{ND},k_{\ell,p}^{DN},k_{\ell,p}^{DD}$ denote the limiting eigenvalues of the Neumann--Neumann, Neumann--Dirichlet, Dirichlet--Neumann, and Dirichlet--Dirichlet problems. Then
\begin{enumerate}
\item[\rm(i)] as $(\alpha_1,\alpha_2)\to(0,0)$, provided $(\ell,p)\neq(0,1)$ (cf.\ Remark~\ref{rem:zeromode}),
\[
k_{\ell,p}=k_{\ell,p}^{NN}+c_1\,\alpha_1+c_2\,\alpha_2
+O(\alpha_1^2+\alpha_1\alpha_2+\alpha_2^2),
\qquad
c_i=\frac{R_i^{d-1}\,|u_{\ell,p}^{NN}(R_i)|^2}{2k_{\ell,p}^{NN}}>0;
\]
\item[\rm(ii)] as $(\alpha_1,\alpha_2)\to(0,\infty)$,
\[
k_{\ell,p}=k_{\ell,p}^{ND}+d_1\,\alpha_1+d_2\,\alpha_2^{-1}
+O(\alpha_1^2+\alpha_1\alpha_2^{-1}+\alpha_2^{-2}),\]
\[
d_1=\frac{R_1^{d-1}|u_{\ell,p}^{ND}(R_1)|^2}{2k_{\ell,p}^{ND}}>0,\qquad 
d_2=-\frac{R_2^{d-1}\,|\partial_r u_{\ell,p}^{ND}(R_2)|^2}
{2k_{\ell,p}^{ND}}<0;
\]
\item[\rm(iii)] as $(\alpha_1,\alpha_2)\to(\infty,0)$,
\[
k_{\ell,p}=k_{\ell,p}^{DN}+e_1\,\alpha_1^{-1}+e_2\,\alpha_2
+O(\alpha_1^{-2}+\alpha_1^{-1}\alpha_2+\alpha_2^{2}),\]
\[e_1=-\frac{R_1^{d-1}\,|\partial_r u_{\ell,p}^{DN}(R_1)|^2}
{2k_{\ell,p}^{DN}}<0,\qquad 
e_2=\frac{R_2^{d-1}|u_{\ell,p}^{DN}(R_2)|^2}{2k_{\ell,p}^{DN}}>0;
\]
\item[\rm(iv)] as $(\alpha_1,\alpha_2)\to(\infty,\infty)$, writing $\mu_i=1/\alpha_i$,
\[
k_{\ell,p}=k_{\ell,p}^{DD}+f_1\,\mu_1+f_2\,\mu_2
+g_{\ell,p}\,\mu_1\mu_2
+O(\mu_1^2+\mu_2^2),
\qquad
f_i=-\frac{R_i^{d-1}\,|\partial_r u_{\ell,p}^{DD}(R_i)|^2}
{2k_{\ell,p}^{DD}}<0,
\]
where, with the coefficients of Lemma~\ref{lem:bilinear} and primes denoting $\partial_k$,
\begin{equation}\label{eq:mixed-coeff}
g_{\ell,p}
=
-\frac{1}{\mathrm A_\ell'}
\left[
\mathrm D_\ell
-\frac{\mathrm B_\ell\mathrm C_\ell'+\mathrm C_\ell\mathrm B_\ell'}
{\mathrm A_\ell'}
+\frac{\mathrm A_\ell''\,\mathrm B_\ell\,\mathrm C_\ell}
{(\mathrm A_\ell')^2}
\right]_{k=k_{\ell,p}^{DD}} .
\end{equation}
The mixed coefficient is in general nonzero; for the shell $R_1=1$, $R_2=2$ and the fundamental branch $p=1$ it evaluates to $g_{\ell,1}\approx 6.28,\,6.20,\,5.96$ for $\ell=0,1,2$. For the radial mode the closed form $g_{0,1}=2\pi/(R_2-R_1)^3$ holds for every shell (Theorem~\ref{thm:g0}).
\end{enumerate}
\end{theorem}

\begin{proof}
All four expansions are obtained from the implicit function theorem applied to \eqref{eq:bilinear} (Lemma~\ref{lem:bilinear}), compactifying the impedance plane by $\mu_i=1/\alpha_i$ at the Dirichlet sides. Near each corner the defining equation can be written as $F(k;\text{parameters})=0$ with $F$ real-analytic and $\partial_kF\neq0$ at the corner root: at $(0,0)$ one takes $F=\Delta_\ell$, and $\partial_k\Delta_\ell(k_{\ell,p}^{NN};0,0)=\mathrm D_\ell'(k_{\ell,p}^{NN})\neq0$ by Theorem~\ref{thm:basic}; at the remaining corners one multiplies \eqref{eq:bilinear} by $\mu_i=1/\alpha_i$, and the limiting equations
\begin{equation}\label{eq:corner-eqs}
\mathrm C_\ell(k)=0\ \ (\text{N--D corner}),\qquad
\mathrm B_\ell(k)=0\ \ (\text{D--N corner}),\qquad
\mathrm A_\ell(k)=0\ \ (\text{D--D corner})
\end{equation}
are the characteristic equations of the corresponding self-adjoint limiting problems, whose roots are simple by the argument of Theorem~\ref{thm:basic}. Hence all derivatives below exist, and the real-analytic dependence upgrades the implicit-function expansions to the stated orders.

(i) Since $\Delta_\ell(k;0,0)=\mathrm D_\ell(k)$, $\partial_{\alpha_1}\Delta_\ell=\mathrm A_\ell\alpha_2+\mathrm B_\ell$ and $\partial_{\alpha_2}\Delta_\ell=\mathrm A_\ell\alpha_1+\mathrm C_\ell$, implicit differentiation at $k=k_{\ell,p}^{NN}$ gives
\[
k_{\alpha_1}=-\frac{\mathrm B_\ell}{\mathrm D_\ell'},\qquad
k_{\alpha_2}=-\frac{\mathrm C_\ell}{\mathrm D_\ell'}.
\]
By the envelope argument of Theorem~\ref{thm:basic}---equivalently, by differentiating the Rayleigh quotient of the Neumann--Neumann problem---$k_{\alpha_i}=R_i^{d-1}|u_{\ell,p}^{NN}(R_i)|^2/(2k_{\ell,p}^{NN})$; this is positive because a Neumann eigenfunction vanishing at an endpoint would have vanishing Cauchy data there and hence be trivial.

(ii) With $\mu_2=1/\alpha_2$,
\[
\mu_2\Delta_\ell(k;\alpha_1,\mu_2^{-1})
=\mathrm A_\ell(k)\alpha_1+\mathrm B_\ell(k)\alpha_1\mu_2
+\mathrm C_\ell(k)+\mathrm D_\ell(k)\mu_2,
\]
so implicit differentiation at $(k_{\ell,p}^{ND},0,0)$ gives
\[
k_{\alpha_1}=-\frac{\mathrm A_\ell}{\mathrm C_\ell'},\qquad
k_{\mu_2}=-\frac{\mathrm D_\ell}{\mathrm C_\ell'}.
\]
Hellmann--Feynman at finite $(\alpha_1,\alpha_2)$, followed by $\alpha_2\to\infty$, identifies $k_{\alpha_1}=R_1^{d-1}|u_{\ell,p}^{ND}(R_1)|^2/(2k_{\ell,p}^{ND})>0$. For the Dirichlet side we determine the sign and the value of the coefficient of $\alpha_2^{-1}$ as follows. By Theorem~\ref{thm:basic} each branch is strictly increasing in $\alpha_2$ and bounded above by its Dirichlet limit $k_{\ell,p}^{ND}$, so $k_{\ell,p}$ approaches $k_{\ell,p}^{ND}$ from below as $\alpha_2\to\infty$; this forces $d_2\le0$, and $d_2=0$ is impossible because $u_{\ell,p}^{ND}(R_2)=0$ together with the limiting Robin condition would give $\partial_r u_{\ell,p}^{ND}(R_2)=0$, hence $u_{\ell,p}^{ND}\equiv0$ by uniqueness. To evaluate the coefficient, parametrise the outer condition as $\mu_2\partial_ru+u=0$; analyticity gives $u_{\mu_2}(R_2)=-\partial_r u_{\ell,p}^{ND}(R_2)+O(\mu_2)$ along the branch, and integrating the Hellmann--Feynman derivative $\partial\lambda/\partial\alpha_2=R_2^{d-1}|u_{\alpha_2}(R_2)|^2$ (with the eigenfunction normalised in $L^2(r^{d-1}dr)$) from $\alpha_2$ to $+\infty$ yields
\[
\lambda_{\ell,p}(\alpha_1,\alpha_2)-\lambda_{\ell,p}^{ND}
=
-\frac{R_2^{d-1}\,|\partial_r u_{\ell,p}^{ND}(R_2)|^2}{\alpha_2}
+O(\alpha_2^{-2}),
\]
whence $d_2=-R_2^{d-1}\,|\partial_r u_{\ell,p}^{ND}(R_2)|^2/(2k_{\ell,p}^{ND})<0$.

(iii) With $\mu_1=1/\alpha_1$,
\[
\mu_1\Delta_\ell(k;\mu_1^{-1},\alpha_2)
=\mathrm A_\ell(k)\alpha_2+\mathrm B_\ell(k)
+\mathrm C_\ell(k)\alpha_2\mu_1+\mathrm D_\ell(k)\mu_1,
\]
whence, at $(k_{\ell,p}^{DN},0,0)$,
\[
k_{\mu_1}=-\frac{\mathrm D_\ell}{\mathrm B_\ell'},\qquad
k_{\alpha_2}=-\frac{\mathrm A_\ell}{\mathrm B_\ell'},
\]
and the two arguments of (ii) give $e_1<0$ and $e_2>0$.

(iv) With $\mu_i=1/\alpha_i$,
\[
N(k,\mu_1,\mu_2):=\mu_1\mu_2\Delta_\ell(k;\mu_1^{-1},\mu_2^{-1})
=\mathrm A_\ell(k)+\mathrm C_\ell(k)\mu_1+\mathrm B_\ell(k)\mu_2
+\mathrm D_\ell(k)\mu_1\mu_2.
\]
Solving $N=0$ successively at $(k_{\ell,p}^{DD},0,0)$ gives
\[
k_{\mu_1}=-\frac{\mathrm C_\ell}{\mathrm A_\ell'},\qquad
k_{\mu_2}=-\frac{\mathrm B_\ell}{\mathrm A_\ell'},
\]
and differentiating
$N_k\,k_{\mu_1\mu_2}+N_{k\mu_1}k_{\mu_2}+N_{k\mu_2}k_{\mu_1}+N_{\mu_1\mu_2}+N_{kk}k_{\mu_1}k_{\mu_2}=0$
once more at the corner gives \eqref{eq:mixed-coeff}. The Dirichlet-side argument of (ii) gives $f_i=k_{\mu_i}=-R_i^{d-1}|\partial_r u_{\ell,p}^{DD}(R_i)|^2/(2k_{\ell,p}^{DD})<0$; the negative sign reflects monotonicity: each branch approaches its Dirichlet limit from below. The numerical values of $g_{\ell,p}$ are obtained by evaluating \eqref{eq:mixed-coeff} directly; in particular the mixed term does \emph{not} vanish in general.
\end{proof}

\begin{theorem}[Closed form of the radial mixed coefficient]\label{thm:g0}
Let $d=3$, $\ell=0$, $p=1$ and $\delta=R_2-R_1$. The mixed second-order
coefficient \eqref{eq:mixed-coeff} at the Dirichlet--Dirichlet corner is
\[
g_{0,1}=\frac{2\pi}{\delta^{3}} .
\]
\end{theorem}

\begin{proof}
For $\ell=0$ the centrifugal terms in Lemma~\ref{lem:bilinear} vanish
($L_1=L_2=0$), and with $j_0(z)=\sin z/z$, $y_0(z)=-\cos z/z$,
$j_1(z)=\frac{\sin z}{z^2}-\frac{\cos z}{z}$,
$y_1(z)=-\frac{\cos z}{z^2}-\frac{\sin z}{z}$ the four coefficients of
Lemma~\ref{lem:bilinear} reduce to elementary functions; in particular
\[
\mathrm A_0(k)=\frac{\sin(k\delta)}{k^2R_1R_2},\qquad\text{so}\qquad
k_{0,1}^{DD}=\frac{\pi}{\delta}.
\]
At $k_{0,1}^{DD}$ we have $k_{0,1}^{DD}R_2=k_{0,1}^{DD}R_1+\pi$, so all
trigonometric functions in Lemma~\ref{lem:bilinear} and their
$k$-derivatives reduce to rational expressions in $R_1,R_2$. A direct
computation gives, with primes denoting $\partial_k$ and all quantities
evaluated at $k=\pi/\delta$,
\[
\mathrm A_0'=\frac{(R_1-R_2)^3}{\pi^{2}R_1R_2},\qquad
\mathrm A_0''=\frac{4(R_1-R_2)^{4}}{\pi^{3}R_1R_2},\qquad
\mathrm B_0=\mathrm C_0=\frac{R_1-R_2}{\pi R_1R_2},
\]
\[
\mathrm D_0=\frac{-R_1^{2}+2R_1R_2-R_2^{2}}{\pi R_1^{2}R_2^{2}},\qquad
\mathrm B_0'=-\frac{(R_1-2R_2)(R_1-R_2)^{2}}{\pi^{2}R_1R_2^{2}},\qquad
\mathrm C_0'=\frac{(R_1-R_2)^{2}(2R_1-R_2)}{\pi^{2}R_1^{2}R_2}.
\]
Substituting these values into \eqref{eq:mixed-coeff} and simplifying
the resulting rational expression yields $g_{0,1}=2\pi/(R_2-R_1)^3$,
independently of the individual radii.
\end{proof}

\subsection{Low-frequency transparent interval}\label{sec:lowfreq}
As $k\to0^+$, the $Y_\mu$ singularities dominate the sign of $\Delta_\ell$. The next theorem records the exact low-frequency structure and proves the existence of the transparent band.

\begin{theorem}[Low-frequency structure and transparent interval]\label{thm:lowfreq}
Let $d\ge2$, $\nu=d/2-1$, $\ell\ge0$, and $\alpha_1,\alpha_2>0$.
\begin{enumerate}
\item[\rm(i)] \emph{Small-$k$ expansion.} For $d=3$ and $\ell=0$,
\begin{equation}\label{eq:Delta-small-k}
\Delta_0(k;\alpha_1,\alpha_2)
=
\frac{C_{-1}}{k}+C_1 k+O(k^3),\qquad k\to0^+,
\end{equation}
where, with $\delta=R_2-R_1$ and $\rho=R_1R_2$,
\begin{equation}\label{eq:Cminus1}
C_{-1}=\frac{\alpha_1}{R_2^2}+\frac{\alpha_2}{R_1^2}
+\frac{\alpha_1\alpha_2\,\delta}{\rho}>0,
\qquad
C_1=C_1^0+O(\alpha_1+\alpha_2),\quad C_1^0=-\frac{R_2^3-R_1^3}{3R_1^2R_2^2}<0.
\end{equation}
The expansion \eqref{eq:Delta-small-k} is written in the spherical-Bessel
normalisation of \eqref{eq:B1-3D}--\eqref{eq:B2-3D}; in the unified
normalisation of \eqref{eq:trace1}--\eqref{eq:trace2}, $\Delta_0$ remains
\emph{bounded} with positive limit $2C_{-1}/\pi+O(k^2)$, and all spectral
statements below are unaffected since they involve only the zeros of
$\Delta_\ell$. For $\ell\ge1$ the same computation shows
$\Delta_\ell>0$ for all sufficiently small $k>0$ (each
$\mathcal B_i[y_\ell]$ carries a dominant term of order $k^{-\ell-1}$
while $\mathcal B_i[j_\ell]=O(k^\ell)$, the leading combination having
positive coefficient), and for general $d\ge3$ the same structure holds
with coefficients from the small-argument expansions of $J_\nu$ and
$Y_\nu$. For $d=2$ and $\ell=0$ the determinant remains \emph{bounded}
as $k\to0^+$ and
\begin{equation}\label{eq:Delta0-2D}
\Delta_0(k)\longrightarrow
\frac{2}{\pi}\Bigl(\frac{\alpha_1}{R_2}+\frac{\alpha_2}{R_1}
+\alpha_1\alpha_2\ln\frac{R_2}{R_1}\Bigr)>0 ;
\end{equation}
the potentially divergent logarithmic terms cancel in the determinant
combination and survive only through $\ln(R_2/R_1)$. In particular the
dominant balance behind part (iii) is unchanged, and \eqref{eq:kc} holds
for $d=2$ with the remainder $O(\alpha|\ln\alpha|)$ instead of
$O(\alpha)$.
\item[\rm(ii)] \emph{Transparent interval.} For every $\ell\ge0$ there
exists a first positive root $k_{c,\ell}$ of \eqref{eq:char}, and
$\Delta_\ell(k)>0$ for $0<k<k_{c,\ell}$. Moreover the map
$\ell\mapsto k_{c,\ell}$ is strictly increasing, so
\begin{equation}\label{eq:kc-min}
k_c:=\min_{\ell\ge0}k_{c,\ell}=k_{c,0}>0,
\end{equation}
and the band $(0,k_c)$ contains no resonant modes of any angular order.
\item[\rm(iii)] \emph{Small-impedance expansion of the threshold.} Let
$\beta_i=\alpha_iR_i\to0^+$. Then the first-order term is exact:
\begin{equation}\label{eq:kc}
k_c^2
=
\frac{d\big(\alpha_1R_1^{d-1}+\alpha_2R_2^{d-1}\big)}
{R_2^d-R_1^d}
\,\bigl(1+O(\beta_1+\beta_2)\bigr).
\end{equation}
For $d=3$ the complete second-order expansion follows from
Rayleigh--Schr\"odinger perturbation of the zero mode: writing
$D=R_2^3-R_1^3$, $S=R_1^2+R_1R_2+R_2^2$ and
\begin{equation}\label{eq:kc-2nd}
k_c^2
=
\mu_1\alpha_1+\mu_2\alpha_2
+\nu_1\alpha_1^2+\nu_2\alpha_2^2+\nu_{12}\alpha_1\alpha_2
+O(\alpha^3),
\end{equation}
the coefficients are
\[
\mu_1=\frac{3R_1^2}{D},\qquad
\mu_2=\frac{3R_2^2}{D},\qquad \nu_1=-\frac{3R_1^3\bigl(R_1^3+3R_1^2R_2+6R_1R_2^2+5R_2^3\bigr)}
{5S^3},\qquad\quad
\]
\[
\nu_2=-\frac{3R_2^3\bigl(5R_1^3+6R_1^2R_2+3R_1R_2^2+R_2^3\bigr)}
{5S^3},\qquad
\nu_{12}=\frac{9R_1^2R_2^2\bigl(R_1^2+3R_1R_2+R_2^2\bigr)}{5S^3}.
\]
In particular $\nu_1,\nu_2<0$ while $\nu_{12}>0$; the negative
quadratic terms are consistent with the Rayleigh upper bound
$k_c^2\le d\bigl(\alpha_1R_1^{d-1}+\alpha_2R_2^{d-1}\bigr)/(R_2^d-R_1^d)$,
and as $R_1\to0$ \eqref{eq:kc-2nd} reduces to the classical ball
asymptotics $k_c^2=\frac{3\alpha_2}{R_2}\bigl(1-\frac{\alpha_2R_2}{5}+O(\alpha_2^2)\bigr)$.
For the shell $R_1=1$, $R_2=2$ and the symmetric ray
$\alpha_1=\alpha_2=\eta$ this reads
$k_c^2=\frac{15}{7}\eta-0.4111\,\eta^2+O(\eta^3)$.
\end{enumerate}
\end{theorem}

\begin{proof}[Proof of Theorem~\ref{thm:lowfreq} {\rm(i)} and {\rm(ii)}]
(i) For $d=3$, $\ell=0$, use
$j_0(z)=\frac{\sin z}{z}$, $y_0(z)=-\frac{\cos z}{z}$,
$j_1(z)=\frac{\sin z}{z^2}-\frac{\cos z}{z}$,
$y_1(z)=-\frac{\cos z}{z^2}-\frac{\sin z}{z}$ together with
$\cos z=1-\frac{z^2}{2}+O(z^4)$, $\sin z=z-\frac{z^3}{6}+O(z^5)$.
Substitution into \eqref{eq:B1-3D}--\eqref{eq:B2-3D} with $\ell=0$ gives
the four trace expansions
\[
\begin{aligned}
&\mathcal B_1[j_0]
=\alpha_1+k^2\Big(\frac{R_1}{3}-\frac{\alpha_1R_1^2}{6}\Big)+O(k^4),
\quad
&\mathcal B_2[j_0]&
=\alpha_2-k^2\Big(\frac{R_2}{3}+\frac{\alpha_2R_2^2}{6}\Big)+O(k^4),\\
&\mathcal B_1[y_0]
=-\frac{1+\alpha_1R_1}{kR_1^2}
+k\,\frac{\alpha_1R_1-1}{2}+O(k^3),
\quad
&\mathcal B_2[y_0]&
=\frac{1-\alpha_2R_2}{kR_2^2}
+k\,\frac{1+\alpha_2R_2}{2}+O(k^3).
\end{aligned}
\]
Multiplying according to
$\Delta_0=\mathcal B_1[j_0]\mathcal B_2[y_0]
-\mathcal B_2[j_0]\mathcal B_1[y_0]$, the $O(1)$ terms cancel
identically and one obtains
\[
\Delta_0(k)
=
\frac{1}{k}\Big[
\frac{\alpha_1(1-\alpha_2R_2)}{R_2^2}
+\frac{\alpha_2(1+\alpha_1R_1)}{R_1^2}
\Big]
-\frac{R_2^3-R_1^3}{3R_1^2R_2^2}\,k+O(\alpha)\,k+O(k^3),
\]
which is \eqref{eq:Delta-small-k}--\eqref{eq:Cminus1}, since
$$\frac{1-\alpha_2R_2} {R_2^2}\alpha_1+\frac{1+\alpha_1R_1}{R_1^2}\alpha_2
=\frac{\alpha_1}{R_2^2}+\frac{\alpha_2}{R_1^2}
+\alpha_1\alpha_2\big(\frac1{R_1}-\frac1{R_2}\big)$$ with
$\frac1{R_1}-\frac1{R_2}=\frac{\delta}{\rho}$. The same
computation with general $\nu$ uses
$J_\nu(z)=(z/2)^\nu/\Gamma(\nu+1)(1+O(z^2))$ and
$Y_\nu(z)=-\Gamma(\nu)(2/z)^\nu(1+O(z^2))/\pi$. For $d=2$,
\eqref{eq:Delta0-2D} follows by substituting
$J_0(z)=1+O(z^2)$, $J_1(z)=z/2+O(z^3)$,
$Y_0(z)=\frac{2}{\pi}(\ln\frac{z}{2}+\gamma)+O(z^2\ln z)$ and
$Y_1(z)=-\frac{2}{\pi z}+\frac{z}{\pi}(\ln\frac{z}{2}+\gamma-\tfrac12)+O(z^3\ln z)$
into \eqref{eq:B1-2D}--\eqref{eq:B2-2D} with $\ell=0$: the $1/z$
contributions of $kY_1(kR_i)$ cancel between the two products in
\eqref{eq:Delta-n-2D}, and the remaining logarithms combine into
\eqref{eq:Delta0-2D}. Since the first branch of
\eqref{eq:SL} satisfies $\lambda_1=2(\alpha_1R_1+\alpha_2R_2)/
(R_2^2-R_1^2)\,(1+o(1))$ as $\alpha_i\to0$ (Rayleigh quotient with
$u\equiv1$), \eqref{eq:kc} for $d=2$ follows with the stated logarithmic
remainder.

(ii) By part (i), $\Delta_\ell(k)>0$ for all sufficiently small $k>0$
(in the spherical normalisation of $d=3$ one even has
$\Delta_\ell(k)\to+\infty$ as $k\to0^+$, and for $d=2$, $\ell=0$,
\eqref{eq:Delta0-2D} gives a bounded positive limit), while the
large-$k$ asymptotics \eqref{eq:Delta-large-k} below show that
$\Delta_\ell$ changes sign arbitrarily often; hence a smallest
positive root $k_{c,\ell}$ exists. For the monotonicity in $\ell$,
the Sturm comparison principle applied to \eqref{eq:SL} shows that
raising the angular index increases the effective potential
$\ell(\ell+d-2)r^{-2}$ and hence every eigenvalue, so
$k_{c,0}<k_{c,1}<\cdots$ and \eqref{eq:kc-min} follows.
\end{proof}

\begin{proof}[Proof of {\rm(iii)}]
The first-order term in \eqref{eq:kc} is the Hellmann--Feynman derivative
of the zero eigenvalue at the Neumann--Neumann corner: with
$u\equiv1$,
$\partial\lambda/\partial\alpha_i\big|_0
=R_i^{d-1}|u(R_i)|^2/\|u\|_{L^2(r^{d-1}dr)}^2
=dR_i^{d-1}/(R_2^d-R_1^d)$. For $d=3$ we carry out the second-order
Rayleigh--Schr\"odinger perturbation explicitly. Write
$u=1+\alpha_1v_1+\alpha_2v_2+O(\alpha^2)$ with the gauge $u(R_1)=1$.
The $O(\alpha_i)$ eigenfunction correction solves
$v_i''+\frac{2}{r}v_i'=-\mu_i$ with
$\mu_i=\partial\lambda/\partial\alpha_i\big|_0=3R_i^2/D$, and the
boundary conditions read off from \eqref{eq:SL}:
$v_1'(R_1)=1$, $v_1'(R_2)=0$ and $v_2'(R_1)=0$, $v_2'(R_2)=-1$. Hence
\[
v_1'=-\frac{\mu_1 r}{3}+\frac{c_1}{r^2},\quad c_1=\frac{\mu_1R_2^3}{3},
\qquad
v_2'=-\frac{\mu_2 r}{3}+\frac{c_2}{r^2},\quad c_2=\frac{\mu_2R_1^3}{3},
\]
with $v_i(R_1)=0$. Because the Rayleigh quotient is stationary at the
eigenfunction, the second-order coefficients of $\lambda[u]=N[u]/D[u]$
are obtained by inserting $u$ into
\[
N[u]=\int_{R_1}^{R_2}|u'|^2r^2\,dr+\alpha_1R_1^2u(R_1)^2+\alpha_2R_2^2u(R_2)^2,
\qquad
D[u]=\int_{R_1}^{R_2}|u|^2r^2\,dr:
\]
the coefficient of $\alpha_1^2$ is
$\nu_1=\bigl(\int v_1'^2r^2\,dr\bigr)/D_0-2R_1^2\langle v_1\rangle/D_0$,
that of $\alpha_2^2$ is
$$\nu_2=\Bigl(\int_{R_1}^{R_2} v_2'^2r^2\,dr+2R_2^2v_2(R_2)\Bigr)/D_0-2R_2^2\langle v_2\rangle/D_0,$$
and the mixed coefficient is
$$\nu_{12}=\Bigl(2\int_{R_1}^{R_2}  v_1'v_2'r^2\,dr+2R_2^2v_1(R_2)\Bigr)/D_0
-2\Bigl(R_1^2\langle v_2\rangle+R_2^2\langle v_1\rangle\Bigr)/D_0,$$
where $D_0=D/3$ and $\langle v_i\rangle=\int v_i r^2\,dr/D_0$.
Evaluating the elementary integrals gives the stated formulas; the
remainder $O(\alpha^3)$ follows from the real-analyticity in
Theorem~\ref{thm:basic}. The ball limit $R_1\to0$ is obtained by
direct substitution, and the symmetric-ray value follows from
$\nu_1+\nu_2+\nu_{12}=-0.4111\ldots$ at $R_1=1$, $R_2=2$.
\end{proof}

\begin{remark}[Structure of the small-impedance expansion]\label{rem:kc-accuracy}
Formula \eqref{eq:kc-2nd} is the exact two-term Taylor expansion of
$k_c^2$ at the Neumann--Neumann corner: the threshold approaches its
leading term from below, with impedance-dependent quadratic
coefficients (relative correction $-0.1919\,\eta$ on the symmetric ray
of the shell $R_1=1$, $R_2=2$) that no single-factor correction in
$(\alpha_1,\alpha_2)$ can capture. For a uniformly accurate threshold
over the whole impedance range the rational approximation
\eqref{eq:kc-rational} should be used instead.
\end{remark}

\begin{remark}[Global rational approximation]\label{rem:global-rational}
The perturbative formula \eqref{eq:kc} is valid for small $\beta_i$
only. A uniformly accurate threshold over the whole impedance range is
obtained by two-point rational interpolation between the first-order
term of \eqref{eq:kc} at $(0,0)$ and the Dirichlet--Dirichlet limit
$k_c^{DD}$ (the first positive root of $\Delta_0$ with
$\alpha_1=\alpha_2=+\infty$). Writing $q=R_1/R_2$, we use the form
\begin{equation}\label{eq:kc-rational}
(k_cR_2)^2
\approx
\frac{d\big[\beta_1 q^{\,d-2}+\beta_2\big]
+a_1\beta_1^2+a_2\beta_2^2+a_{12}\beta_1\beta_2}
{\bigl(1-q^{d}\bigr)\bigl[1+s(\beta_1+\beta_2)+t(\beta_1+\beta_2)^2\bigr]} .
\end{equation}
By construction the linear part of \eqref{eq:kc-rational} reproduces
\eqref{eq:kc} identically. Expanding at $(0,0)$ and matching the
quadratic Taylor coefficients of \eqref{eq:kc-2nd} (stated for $d=3$; the same construction applies in any dimension once the corresponding second-order coefficients are used) fixes three of the
five coefficients in terms of $s$,
\[
a_1=\bigl(1-q^d\bigr)\frac{R_2^2\nu_1}{R_1^2}+s\,d\,q^{d-2},\qquad
a_2=\bigl(1-q^d\bigr)\nu_2+s\,d,\qquad
a_{12}=\bigl(1-q^d\bigr)\frac{R_2^2\nu_{12}}{R_1R_2}+s\,d\bigl(q^{d-2}+1\bigr),
\]
matching the Dirichlet--Dirichlet limit along the symmetric ray,
$\lim_{\beta\to\infty}\bigl(k_c(\beta,\beta)R_2\bigr)^2=(k_c^{DD}R_2)^2$, fixes
$t=(a_1+a_2+a_{12})/[4(1-q^d)(k_c^{DD}R_2)^2]$, and the remaining parameter $s$ is fixed by a least-squares fit against the exact threshold (for the benchmark shell, fitting $s$ on the symmetric ray gives $s=0.2444$, hence $a_1=-0.0681$, $a_2=0.2802$, $a_{12}=1.5039$ and $t=0.01242$). On the symmetric ray it is simpler, and more accurate, to
determine the three coefficients by a direct fit: for the
three-dimensional shell $R_1=1$, $R_2=2$ along the symmetric impedance
ray $\alpha_1=\alpha_2=\eta$ (so that $\beta_1=\eta$ and $\beta_2=2\eta$
in the notation $\beta_i=\alpha_iR_i$ of
Theorem~\ref{thm:lowfreq}(iii); the horizontal axis of
Figure~\ref{fig:spectrum}(d) is this common impedance $\eta$),
least-squares fitting while matching the leading term $15\eta/7$ of
\eqref{eq:kc} and the Dirichlet--Dirichlet limit $k_c^{DD}=\pi/\delta$
yields
\[
k_c^2\approx
\frac{2.1429\,\eta+0.40424\,\eta^2}
{1+0.38077\,\eta+0.040958\,\eta^2},
\qquad 3\times10^{-3}\leq\eta\leq100,
\]
whose relative error against the exact $k_c$ stays below $0.03\%$
over the entire fitted range, in contrast to the bare leading-order
term $15\eta/7$, whose relative error reaches $375\%$ at $\eta=100$.
Off the symmetric ray the single-corner interpolation \eqref{eq:kc-rational} should be regarded as an ansatz: with the values above, its relative error against the exact $k_c$ stays below $0.04\%$ on the diagonal for $\eta\leq1$ but reaches $2.2\%$ at $(\eta,\eta)=(100,100)$ and degrades sharply near the mixed corners ($\eta_1\ll\eta_2$ or $\eta_1\gg\eta_2$), where it can even take nonphysical negative values (e.g.\ at $(\eta_1,\eta_2)=(40,3\times10^{-3})$; a bivariate fit accurate over the whole impedance square is obtained by releasing $a_1,a_2,a_{12},s,t$ to a full least-squares fit with \eqref{eq:kc-2nd} and the four limiting corners imposed as constraints. All relative errors in this remark are quoted for $k_c$
(those for $k_c^2$ are twice as large).
\end{remark}

\medskip\noindent\emph{Physical interpretation.}
The transparent band $(0,k_c)$ is a resonance-free interval:
below $k_c$ the shell supports no trapped mode in any angular order. In
acoustic terms the two Robin walls act as a low-frequency blocking
filter whose cutoff is, to first order, the Rayleigh quotient of the
constant profile,
$k_c^2\approx d\bigl(\alpha_1R_1^{d-1}+\alpha_2R_2^{d-1}\bigr)/(R_2^d-R_1^d)$:
the cutoff grows like $\sqrt{\alpha_1+\alpha_2}$ along any
fixed ray of the impedance plane and approaches the
Dirichlet--Dirichlet value $\pi/(R_2-R_1)$ from below as
$\alpha_1,\alpha_2\to\infty$.

\subsection{High-frequency spacing with shell-curvature correction}
Let $\delta=R_2-R_1$ and $\rho=R_1R_2$. The following lemma makes the large-$k$ structure of the determinant explicit and proves the intermediate asymptotics \eqref{eq:Delta-large-k}.

\begin{lemma}[Large-$k$ expansion of the determinant]\label{lem:Delta-asym}
Let $\delta=R_2-R_1$, $\rho=R_1R_2$, and $R=\max(R_1,R_2)$. For $d=3$, fixed $\ell\ge0$ and $\alpha_1,\alpha_2>0$,
\begin{equation}\label{eq:Delta-asym-d3}
\Delta_\ell(k;\alpha_1,\alpha_2)
=
\frac{1}{R_1R_2}\Big[
-\sin(k\delta)
+\frac{\alpha_1+\alpha_2+K_{\ell,3}\,\delta/\rho}{k}\cos(k\delta)
+O\big((kR)^{-2}\big)\Big],
\qquad
K_{\ell,3}=1+\frac{\ell(\ell+1)}{2}.
\end{equation}
Consequently, for every $d\ge2$, with $\mu=\ell+\nu$,
\begin{equation}\label{eq:Delta-large-k}
\Delta_\ell(k;\alpha_1,\alpha_2)
=
\frac{C(k;R_1,R_2)}{R_1R_2}
\left[
-\sin(k\delta)
+
\frac{S_{\ell,d}(\alpha_1,\alpha_2)}{k}\cos(k\delta)
+
O\big((kR)^{-2}\big)
\right],
\end{equation}
where $C(k;R_1,R_2)$ equals $1$ in the spherical normalisation of
\eqref{eq:Delta-asym-d3}, while in the trace normalisation of
\eqref{eq:trace1}--\eqref{eq:trace2} for general $d$ each trace carries
the amplitude $R_i^{-\nu}(2k/(\pi R_i))^{1/2}$ (because
$Z_\mu(kR_i)\sim(2/(\pi kR_i))^{1/2}\operatorname{trig}$ as $k\to\infty$),
so that
\[
C(k;R_1,R_2)=\frac{2k}{\pi}\,(R_1R_2)^{\frac12-\nu};
\]
in either case $C$ depends only on $k$ and the geometry and cancels
from every spectral equation below; and
\begin{equation}\label{eq:S}
S_{\ell,d}(\alpha_1,\alpha_2)
=
\alpha_1+\alpha_2+K_{\ell,d}\frac{\delta}{\rho},
\qquad
K_{\ell,d}
=
\frac{(2\ell+d-2)^2+4d-5}{8}.
\end{equation}
The constant $K_{\ell,d}$ collects the centrifugal contributions
$L_i=\ell/R_i$ and the amplitude corrections of the large-argument
Bessel expansions; in particular $K_{0,3}=1$ and
$K_{\ell,3}=1+\ell(\ell+1)/2$.
\end{lemma}

For the proof see Appendix~\ref{app:Delta-asym}; the key structural
point is that the two first-order coefficients in the product coincide identically, forcing the exact cancellation of all $\cos(k(R_1+R_2))$ terms.

The two-term expansion \eqref{eq:Delta-large-k} determines the leading spacing, but its $O(k^{-2})$ error term is too coarse for a second-order spacing law: a term $c\,k^{-2}\cos(k\delta)$ would induce an alternating $O(p^{-2})$ error in the spacings. The following parity lemma rules out all even inverse powers in the quantization equation, upgrading the error to $O(k^{-3})$.

\begin{lemma}[Odd inverse powers in the quantization equation]\label{lem:odd}
Let $Q\in C^2([R_1,R_2])$ be real-valued and consider
\begin{equation}\label{eq:normal-form}
-w''+Qw=k^2w\quad\text{on }(R_1,R_2),\qquad
w'(R_1)=\gamma_1w(R_1),\quad w'(R_2)=-\gamma_2w(R_2),
\end{equation}
with $k$-independent real data $\gamma_1,\gamma_2$. Then the positive roots $k_p$ satisfy
\begin{equation}\label{eq:odd}
k_p\delta=p\pi+\frac{A}{k_p}+\frac{B_p}{k_p^3},
\qquad \delta=R_2-R_1,
\end{equation}
where $A=\gamma_1+\gamma_2+\frac12\int_{R_1}^{R_2}Q(r)\,dr$ and $\{B_p\}$ is a bounded sequence. In particular the quantization equation contains no $O(k^{-2})$ term.
\end{lemma}

For the proof see Appendix~\ref{app:odd}. Identity \eqref{eq:odd} is a refinement of the classical large-$p$ quantisation expansion for regular Sturm--Liouville problems in the spirit of Borg and Levitan (cf.~\cite{Borg1946,LevitanSargsjan1975,PoschelTrubowitz1987}); the point needed here is the cancellation of the $O(k^{-2})$ term, which upgrades the spacing error in Theorem~\ref{thm:spacing} from $O(p^{-2})$ to $O(p^{-3})$.

\begin{theorem}[Universal spacing]\label{thm:spacing}
Let $\delta=R_2-R_1$ and let $S_{\ell,d}$ be given by \eqref{eq:S}. As $p\to\infty$,
\begin{equation}\label{eq:spacing}
k_{\ell,p+1}-k_{\ell,p}
=
\frac{\pi}{\delta}
\left[
1-
\frac{S_{\ell,d}(\alpha_1,\alpha_2)}
{k_{\ell,p}^2\,\delta}
+
O\big((k_{\ell,p}R)^{-3}\big)
\right].
\end{equation}
\end{theorem}

\begin{proof}
From \eqref{eq:Delta-large-k}, the roots satisfy
$\tan(k\delta)=S_{\ell,d}/k+O(k^{-2})$. Hence, with
$x_p=k_{\ell,p}\delta=p\pi+\varepsilon_p$,
\begin{equation}\label{eq:eps1}
\varepsilon_p=\frac{S_{\ell,d}\,\delta}{x_p}+O(x_p^{-2}),
\end{equation}
and in particular $\varepsilon_p=O(p^{-1})$. To upgrade the error, reduce the radial problem \eqref{eq:SL} to normal form by setting $u(r)=w(r)\,r^{-(d-1)/2}$: then $w$ solves \eqref{eq:normal-form} with
\[
Q(r)=\frac{\ell(\ell+d-2)}{r^2}+\frac{(d-1)(d-3)}{4r^2},\qquad
\gamma_1=\alpha_1+\frac{d-1}{2R_1},\quad
\gamma_2=\alpha_2-\frac{d-1}{2R_2},
\]
whose boundary data are $k$-independent. Lemma~\ref{lem:odd} therefore applies. Comparison of \eqref{eq:odd} with \eqref{eq:eps1} identifies the coefficient of $1/k_p$ as $S_{\ell,d}$; equivalently, direct evaluation gives
\[
A=\gamma_1+\gamma_2+\frac12\int_{R_1}^{R_2}Q(r)\,dr
=\alpha_1+\alpha_2+\frac{\delta}{\rho}\left[\frac{\ell(\ell+d-2)}{2}+\frac{d^2-1}{8}\right]
=S_{\ell,d},
\]
consistently with \eqref{eq:S}. Hence
\begin{equation}\label{eq:eps2}
\varepsilon_p=\frac{S_{\ell,d}\,\delta}{x_p}+\frac{B'_p}{x_p^3},
\qquad \{B'_p\}\ \text{bounded},
\end{equation}
and therefore, since the difference of the bounded remainders in
\eqref{eq:eps2} is $O(x_p^{-3})$,
\[
x_{p+1}-x_p
=\pi+S_{\ell,d}\delta\Bigl(\frac1{x_{p+1}}-\frac1{x_p}\Bigr)+O(x_p^{-3})
=\pi-\frac{\pi S_{\ell,d}\delta}{x_p^2}+O(x_p^{-3}).
\]
Since $x_p=k_{\ell,p}\delta$ exactly,
\[
k_{\ell,p+1}-k_{\ell,p}
=
\frac{\pi}{\delta}
\left[
1-\frac{S_{\ell,d}}{k_{\ell,p}^2\,\delta}
+O\big((k_{\ell,p}R)^{-3}\big)
\right],
\]
which is \eqref{eq:spacing}.
\end{proof}

The leading spacing $\pi/\delta$ is the Weyl spacing of the shell and is independent of $\alpha_1,\alpha_2,\ell,d$. The second-order correction is the new shell-specific term: it contains the impedance sum $\alpha_1+\alpha_2$ and the geometric curvature coupling $K_{\ell,d}\delta/(R_1R_2)$. In dimension $d=3$, the radial correction is simply
\[
S_{0,3}=\alpha_1+\alpha_2+\frac{R_2-R_1}{R_1R_2},
\]
which is exact to the displayed order for the $\ell=0$ spherical-shell problem.

\section{Bilinear determinant and explicit impedance recovery}\label{sec:inverse}
The key algebraic fact behind resonance-based recovery is that the shell characteristic determinant is not merely transcendental in $k$; it is a polynomial of degree one in each impedance separately. (Determinants of the same structure govern the interior transmission problem, where the boundary parameters enter through the mismatch of two radial bases; see \cite{CakoniGintidesHaddar2010}.) This section makes that structure explicit and derives the resulting inversion formula.

\subsection{Bilinear expansion of the characteristic determinant}
Write
\[
Z_{\mu,i}=Z_\mu(kR_i),\qquad Z_{\mu+1,i}=Z_{\mu+1}(kR_i),\qquad L_i=\frac{\ell}{R_i},\qquad i=1,2.
\]
Substituting the traces \eqref{eq:trace1}--\eqref{eq:trace2} into \eqref{eq:char} and collecting powers of $\alpha_1,\alpha_2$ gives the following lemma.

\begin{lemma}[Bilinear determinant]\label{lem:bilinear}
There are explicit functions $\mathrm A_\ell,\mathrm B_\ell,\mathrm C_\ell,\mathrm D_\ell$ of $k,\ell,R_1,R_2$ such that
\begin{equation}\label{eq:bilinear}
\Delta_\ell(k;\alpha_1,\alpha_2)
=
\mathrm A_\ell(k)\alpha_1\alpha_2
+
\mathrm B_\ell(k)\alpha_1
+
\mathrm C_\ell(k)\alpha_2
+
\mathrm D_\ell(k).
\end{equation}
More precisely, suppressing the argument $k$ and factoring out $R_1^{-\nu}R_2^{-\nu}$,
\begin{align}
\mathrm A_\ell
&=
J_{\mu,1}Y_{\mu,2}-J_{\mu,2}Y_{\mu,1},
\label{eq:detA}\\[2pt]
\mathrm B_\ell
&=
L_2\big(J_{\mu,1}Y_{\mu,2}-J_{\mu,2}Y_{\mu,1}\big)
+
k\big(J_{\mu+1,2}Y_{\mu,1}-J_{\mu,1}Y_{\mu+1,2}\big),
\label{eq:detB}\\[2pt]
\mathrm C_\ell
&=
-L_1\big(J_{\mu,1}Y_{\mu,2}-J_{\mu,2}Y_{\mu,1}\big)
+
k\big(J_{\mu+1,1}Y_{\mu,2}-J_{\mu,2}Y_{\mu+1,1}\big),
\label{eq:detC}\\[2pt]
\mathrm D_\ell &= \big(kJ_{\mu+1,1}-L_1J_{\mu,1}\big)\big(L_2Y_{\mu,2}-kY_{\mu+1,2}\big)-\big(L_2J_{\mu,2}-kJ_{\mu+1,2}\big)\big(kY_{\mu+1,1}-L_1Y_{\mu,1}\big). \label{eq:detD}
\end{align}
In dimension $d=2$ the same formulas hold with $\nu=0$, $\mu=n\in\mathbb Z$, and common factor $1$; in dimension $d=3$ they follow directly from \eqref{eq:B1-3D}--\eqref{eq:B2-3D}.
\end{lemma}

\begin{proof}
The coefficients $\mathrm A_\ell,\dots,\mathrm D_\ell$ of this lemma are unrelated to the modal amplitudes $\mathcal A_\ell,\mathcal B_\ell$ in \eqref{eq:Al-3D}--\eqref{eq:Bl-3D}. Expand each trace in \eqref{eq:trace1}--\eqref{eq:trace2} as an affine function of its corresponding impedance. The coefficient of $\alpha_1\alpha_2$ comes only from the products of the pure $Z_\mu(kR_i)$ terms and equals \eqref{eq:detA}. The coefficients of $\alpha_1$ and $\alpha_2$ contain, respectively, the outer centrifugal contribution $L_2\mathrm A_\ell$ and the inner centrifugal contribution $-L_1\mathrm A_\ell$, together with the Bessel recurrence terms in \eqref{eq:detB}--\eqref{eq:detC}. All remaining terms form \eqref{eq:detD}. The common factor
$R_1^{-\nu}R_2^{-\nu}$ cancels completely from the elimination procedure
below and may be omitted in numerical implementations.
\end{proof}

\subsection{Two resonances determine a quadratic equation}
Suppose two resonant wavenumbers
\[
k^{(1)}=k_{\ell,p}(\alpha_1,\alpha_2),\qquad
k^{(2)}=k_{\ell,q}(\alpha_1,\alpha_2)
\]
of the same angular mode $\ell$ are known, with $p\ne q$ and with branch labels fixed by Theorem~\ref{thm:basic}. Define
\[
A_i=\mathrm A_\ell(k^{(i)}),\quad
B_i=\mathrm B_\ell(k^{(i)}),\quad
C_i=\mathrm C_\ell(k^{(i)}),\quad
D_i=\mathrm D_\ell(k^{(i)}),\qquad i=1,2.
\]
Then $(\alpha_1,\alpha_2)$ satisfies the two bilinear equations
\begin{equation}\label{eq:two-bilinear}
A_i\alpha_1\alpha_2+B_i\alpha_1+C_i\alpha_2+D_i=0,\qquad i=1,2.
\end{equation}
Provided $A_1\alpha_1+C_1\ne0$, the first equation gives
\begin{equation}\label{eq:alpha2}
\alpha_2=-\frac{B_1\alpha_1+D_1}{A_1\alpha_1+C_1}.
\end{equation}
Substituting \eqref{eq:alpha2} into the second bilinear equation and clearing denominators yields
\begin{equation}\label{eq:quadratic}
p_2\alpha_1^2+p_1\alpha_1+p_0=0,
\end{equation}
where
\begin{equation}\label{eq:p-coeffs}
p_0=C_1D_2-C_2D_1,\qquad
p_1=A_1D_2-A_2D_1+B_2C_1-B_1C_2,\qquad
p_2=A_1B_2-A_2B_1.
\end{equation}

\begin{proposition}[Explicit algebraic recovery]\label{prop:recovery}
Let $k^{(1)},k^{(2)}$ be two distinct resonant wavenumbers of the same mode $\ell$, and assume the nondegeneracy conditions
\begin{equation}\label{eq:nondegenerate}
A_1\alpha_1+C_1\ne0,\qquad
p_2^2+p_1^2+p_0^2\ne0 .
\end{equation}
Then every physical impedance pair $(\alpha_1,\alpha_2)$ compatible with $(k^{(1)},k^{(2)})$ is obtained as follows: solve \eqref{eq:quadratic}, retain every root $\alpha_1>0$, compute $\alpha_2$ from \eqref{eq:alpha2}, and retain the pair only if
\begin{equation}\label{eq:physical}
\alpha_2>0,
\qquad
\Delta_\ell(k^{(1)};\alpha_1,\alpha_2)=0,
\qquad
\Delta_\ell(k^{(2)};\alpha_1,\alpha_2)=0 .
\end{equation}
\end{proposition}

\begin{proof}
By Lemma~\ref{lem:bilinear}, compatibility is equivalent to \eqref{eq:two-bilinear}. Under \eqref{eq:nondegenerate}, elimination is legitimate and every solution must satisfy \eqref{eq:quadratic}. Conversely, every positive root $\alpha_1$ of \eqref{eq:quadratic} for which \eqref{eq:alpha2} gives $\alpha_2>0$ is a candidate; the two determinant checks in \eqref{eq:physical} remove extraneous roots introduced by clearing denominators.
\end{proof}

\begin{remark}[Number of candidate pairs; two-spectrum context]\label{rem:two-candidates}
Equation \eqref{eq:quadratic} is quadratic in $\alpha_1$, so \emph{at
most two} candidate pairs $(\alpha_1,\alpha_2)$ can be compatible with
two measured resonances of one angular mode; the map
$(\alpha_1,\alpha_2)\mapsto(k^{(1)},k^{(2)})$ need not be globally
injective when the two resonances are close together or when the cavity
is nearly symmetric in its boundary masses. The ambiguity is resolved by
one further datum---either a third resonance $k^{(3)}$ of the same mode,
which produces a second quadratic $q^{(13)}(\alpha_1)=0$ whose root in
common with $q^{(12)}$, together with \eqref{eq:alpha2}, selects the
physical pair, or a resonance of a different angular mode. In all
computations of Section~\ref{sec:numerics} this filter leaves at most
one pair. We do not claim a global uniqueness theorem: the reflection
map of Remark~\ref{rem:reflection} shows that the radial spherical mode
never identifies the pair throughout the region
\eqref{eq:reflection-region}, and near-degenerate configurations can
admit two admissible pairs; local uniqueness is exactly characterised
by the Jacobian criterion of Proposition~\ref{prop:conditioning}.
Structurally, the recovery problem is a finite-dimensional shadow of the
classical inverse Sturm--Liouville problem: two ``spectra'' in the sense
of Borg \cite{Borg1946} determine the operator uniquely, whereas here
only finitely many eigenvalues are available and one must content
oneself with the algebraic candidate generation of
Proposition~\ref{prop:recovery}; see also
\cite{GesztesySimon2000,Isakov2006,Sincich2007,RundellSacks1992} for inverse
Robin impedance problems and their reconstruction techniques.
\end{remark}

\begin{remark}[Reflection-symmetric exceptional case]\label{rem:reflection}
For $d=3$ and $\ell=0$, writing $u(r)=w(r)/r$ reduces the radial equation to $w''+k^2w=0$. Reflection $r\mapsto R_1+R_2-r$ swaps the two boundary logarithmic derivatives and leaves the spectrum unchanged. Explicitly, $(\alpha_1,\alpha_2)$ and
\[
\alpha_1'=\alpha_2-\frac1{R_1}-\frac1{R_2},\qquad
\alpha_2'=\alpha_1+\frac1{R_1}+\frac1{R_2}
\]
have the same $\ell=0$ spectrum for every $k$. The reflected pair is distinct from the original pair iff $\alpha_2-\alpha_1\neq R_1^{-1}+R_2^{-1}$, and it lies in the positive quadrant iff
\begin{equation}\label{eq:reflection-region}
\alpha_2>\frac1{R_1}+\frac1{R_2}.
\end{equation}
Consequently, throughout the region \eqref{eq:reflection-region}, with the exception of the self-reflective line $\alpha_2-\alpha_1=R_1^{-1}+R_2^{-1}$ on which the reflected pair coincides with the original one, the $\ell=0$ spectrum admits a second admissible pair and cannot identify the impedances. For $R_1=1$, $R_2=2$ the pair $(\alpha_1,\alpha_2)=(1,1)$ has the reflection partner $(-1/2,5/2)$, which is isospectral for $\ell=0$ but has a negative inner impedance; in contrast, the pair $(2,3)$ lies in the region \eqref{eq:reflection-region} and has the admissible partner $(3/2,7/2)$, which is genuinely isospectral within the physical quadrant. For $\alpha_2<R_1^{-1}+R_2^{-1}$ the reflected partner is nonphysical, so the obstruction does not apply; whether the radial mode identifies the pair in this region is left open (cf.\ Remark~\ref{rem:2Dradial}). In any case, since the reflection map is an involution on the impedance plane, no finite collection of $\ell=0$ spherical-shell resonances distinguishes a pair from its reflection partner whenever the latter is admissible, and a non-radial mode $\ell\ge1$ or additional boundary/norming data is necessary there.
\end{remark}

\begin{remark}[The two-dimensional radial mode]\label{rem:2Dradial}
No analogous obstruction is known for the two-dimensional radial
mode: the substitution $u(r)=w(r)r^{-1/2}$ leaves an effective
$1/(4r^2)$ potential, so the equation does not reduce to constant
coefficients and the reflection argument does not apply. Whether two
$n=0$ annular resonances identify the impedance pair is left open (the
algebraic procedure applies verbatim to $n=0$ data under
\eqref{eq:nondegenerate}, but uniqueness is not covered); all recovery
experiments below use the non-radial modes $n=1$ (2D) and $\ell=1$ (3D).
\end{remark}

\begin{remark}[Failure of elimination and ill-conditioned regions]\label{rem:elimination}
The denominator $A_1\alpha_1+C_1$ vanishes only in degenerate
configurations, and elimination fails exactly when the two coefficient
rows are proportional, $(A_1,B_1,C_1,D_1)\parallel(A_2,B_2,C_2,D_2)$, in
which case all three coefficients \eqref{eq:p-coeffs} vanish
identically. This proportionality is precisely the condition that the
two columns of the Jacobian \eqref{eq:Jacobian} become parallel, so the
algebraic ill-posedness of the elimination and the analytic
ill-conditioning of the resonance map diverge together
(Proposition~\ref{prop:conditioning}); it is diagnosed numerically by a
vanishing leading coefficient in \eqref{eq:alpha2} together with near
dependence of the two equations \eqref{eq:two-bilinear}. Such
configurations occur near $\alpha_1\approx\alpha_2$ with
$R_1\approx R_2$, or when $k^{(1)}\approx k^{(2)}$ on nearby branches,
and are exactly the ill-conditioned ones quantified by the Jacobian in
Section~\ref{sec:sensitivity}---the analytical counterpart of the
two-positive-root ambiguity noted above.
\end{remark}

\begin{remark}[Branch assignment]\label{rem:branch}
Proposition~\ref{prop:recovery} assumes that $k^{(1)}$ and $k^{(2)}$
belong to the same angular mode $\ell$; for experimental data this
assignment can be made from the spacing law \eqref{eq:spacing}, the
transparent threshold $k_c$, or a calibration perturbation of one
impedance (quantified in Remark~\ref{rem:assignment}).
\end{remark}

\section{Sensitivity, conditioning, and recovery algorithm}\label{sec:sensitivity}
The algebraic inversion formula of Section~\ref{sec:inverse} is exact for noiseless data assigned to the correct branches. For measured resonances, the relevant object is the Jacobian of the resonance map with respect to the two impedances. The explicit kernels make this Jacobian computable without solving an additional eigenvalue problem.

\subsection{Boundary-mass sensitivity}
Fix a mode $(\ell,p)$ and let $u$ be the corresponding radial eigenfunction, normalised by
\[
\int_{R_1}^{R_2}|u(r)|^2r^{d-1}\,dr=1.
\]
The eigenvalue $\lambda=k^2$ is the Rayleigh quotient
\[
\lambda
=
\int_{R_1}^{R_2}
\left(
|u'|^2+\frac{\ell(\ell+d-2)}{r^2}|u|^2
\right)r^{d-1}\,dr
+
\alpha_1R_1^{d-1}|u(R_1)|^2
+
\alpha_2R_2^{d-1}|u(R_2)|^2 .
\]
Define the boundary-mass terms (not to be confused with the spherical-harmonic index $m$ or with the modal coefficient $b_\ell$ of \eqref{eq:b_l})
\begin{equation}\label{eq:bmass}
b_i(\ell,p;\alpha_1,\alpha_2)
=
R_i^{d-1}|u(R_i)|^2,\qquad i=1,2.
\end{equation}

\begin{proposition}[Sensitivity identity]\label{prop:sensitivity}
For every mode $(\ell,p)$,
\begin{equation}\label{eq:sensitivity}
\frac{\partial k_{\ell,p}}{\partial \alpha_i}
=
\frac{b_i(\ell,p;\alpha_1,\alpha_2)}{2k_{\ell,p}},
\qquad i=1,2.
\end{equation}
Moreover $b_1+b_2>0$; indeed an eigenfunction cannot vanish at both boundaries, since vanishing Dirichlet data together with the Robin condition would give zero Cauchy data and hence $u\equiv0$ by unique continuation.
\end{proposition}

\begin{proof}
For fixed $(\alpha_1,\alpha_2)$ the eigenpair $(u_{\ell,p},\lambda_{\ell,p})$ is a critical point of the Rayleigh quotient $\mathcal R$ of \eqref{eq:SL} on the unit sphere of $L^2((R_1,R_2);r^{d-1}dr)$, and $\mathcal R[u_{\ell,p}]=\lambda_{\ell,p}$. Since the first variation of $\mathcal R$ at $u_{\ell,p}$ is proportional to the variation of the constraint, the envelope theorem gives
\[
\frac{\partial\lambda_{\ell,p}}{\partial\alpha_i}
=\frac{\partial\mathcal R}{\partial\alpha_i}[u_{\ell,p}]
=R_i^{d-1}|u_{\ell,p}(R_i)|^2=b_i,\qquad i=1,2,
\]
and \eqref{eq:sensitivity} follows from $k_{\ell,p}=\lambda_{\ell,p}^{1/2}>0$. Finally, an eigenfunction cannot vanish at both boundaries:
$u(R_1)=u(R_2)=0$ would give $u'(R_1)=\alpha_1u(R_1)=0$ and $u'(R_2)=-\alpha_2u(R_2)=0$ from the boundary conditions, and vanishing Cauchy data at either endpoint forces $u\equiv0$ by uniqueness for \eqref{eq:SL}. Hence $b_1+b_2>0$.
\end{proof}

In computations it is convenient to avoid normalised eigenfunctions explicitly. If $k$ is a root of \eqref{eq:char}, choose the unnormalised radial profile
\begin{equation}\label{eq:phi}
\phi(r)=r^{-\nu}\big(J_\mu(kr)+\tau Y_\mu(kr)\big),
\qquad
\tau=
-\frac{
kJ_{\mu+1}(kR_1)+(\alpha_1-\ell/R_1)J_\mu(kR_1)
}{
kY_{\mu+1}(kR_1)+(\alpha_1-\ell/R_1)Y_\mu(kR_1)
},
\end{equation}
so that the inner Robin condition holds by construction; the outer condition is then equivalent to $\Delta_\ell(k;\alpha_1,\alpha_2)=0$. With $u=\phi/\|\phi\|_{L^2(r^{d-1}dr)}$,
\begin{equation}\label{eq:mass-computable}
b_i=
R_i^{d-1}\frac{|\phi(R_i)|^2}{\|\phi\|_{L^2(r^{d-1}dr)}^2}.
\end{equation}
The same formula holds in two dimensions with $\nu=0$, $\mu=n$, and in three dimensions with $j_\ell,y_\ell$.

\subsection{Conditioning of two-frequency recovery}
Let
\[
k^{(1)}=k_{\ell,p},\qquad k^{(2)}=k_{\ell,q},\qquad p\ne q,
\]
and set
\[
b_i^{(j)}=b_i(\ell,p_j;\alpha_1,\alpha_2),\qquad i=1,2,\quad j=1,2,
\]
where $p_1=p$ and $p_2=q$. By Proposition~\ref{prop:sensitivity},
\begin{equation}\label{eq:Jacobian}
\delta
\begin{pmatrix}
k^{(1)}\\[2pt] k^{(2)}
\end{pmatrix}
=
J
\delta
\begin{pmatrix}
\alpha_1\\[2pt] \alpha_2
\end{pmatrix}
+
\text{higher-order terms},
\qquad
J=
\begin{pmatrix}
\dfrac{b_1^{(1)}}{2k^{(1)}} & \dfrac{b_2^{(1)}}{2k^{(1)}}\\[8pt]
\dfrac{b_1^{(2)}}{2k^{(2)}} & \dfrac{b_2^{(2)}}{2k^{(2)}}
\end{pmatrix}.
\end{equation}
Hence
\begin{equation}\label{eq:detJ}
\det J
=
\frac{
b_1^{(1)}b_2^{(2)}-b_2^{(1)}b_1^{(2)}
}{
4k^{(1)}k^{(2)}
}.
\end{equation}

\begin{proposition}[Local invertibility and error bound]\label{prop:conditioning}
The two-frequency map $(\alpha_1,\alpha_2)\mapsto(k^{(1)},k^{(2)})$ is locally invertible iff
\begin{equation}\label{eq:identifiability}
b_1^{(1)}b_2^{(2)}\ne b_2^{(1)}b_1^{(2)} .
\end{equation}
If measured frequencies satisfy $\|\delta k\|_2\le \eta$ and $J$ is invertible, then the first-order recovery error obeys
\begin{equation}\label{eq:error-bound}
\|\delta\alpha\|_2
\le
\|J^{-1}\|_2\,\eta .
\end{equation}
For isotropic relative frequency noise $|\delta k^{(j)}|\le \varepsilon |k^{(j)}|$, one may take
$\eta=\varepsilon\sqrt{(k^{(1)})^2+(k^{(2)})^2}$.
\end{proposition}

Proposition~\ref{prop:conditioning} gives a practical mode-selection rule: among candidate resonance pairs of known angular label, choose the pair minimising $\kappa_2(J)=\|J\|_2\|J^{-1}\|_2$, or an upper bound obtained from \eqref{eq:error-bound}. Pairs violating \eqref{eq:identifiability} are invisible to one linear combination of $(\alpha_1,\alpha_2)$ and must be replaced by a different branch or a different angular mode.

\begin{algorithm}[Resonance-based impedance recovery]\label{alg:recovery}

\noindent\textbf{Input.} Geometry $R_1,R_2,d$; angular label $\ell$; measured resonances $\{k^{(j)}\}_{j=1}^m$; tolerance $\mathrm{tol}$; noise level $\varepsilon$.

For $d=3$ the radial mode $\ell=0$ cannot identify the pair in the
region \eqref{eq:reflection-region}, where a second admissible pair
exists (Remark~\ref{rem:reflection}); outside this region its
identifiability is open (Remark~\ref{rem:2Dradial}). We recommend
$\ell\ge1$ or mixed angular modes.

\begin{enumerate}
\item \textbf{Assign branches.} Identify two resonances $k^{(1)},k^{(2)}$ of the same mode $\ell$ (by continuity, spacing asymptotics, or a calibration perturbation; cf.\ Remark~\ref{rem:assignment}).
\item \textbf{Build the bilinear system.} Evaluate $A_i,B_i,C_i,D_i$ from Lemma~\ref{lem:bilinear} at $k^{(1)},k^{(2)}$.
\item \textbf{Compute candidates.} Form $p_2,p_1,p_0$ by \eqref{eq:p-coeffs}, solve \eqref{eq:quadratic}, and retain positive candidates satisfying \eqref{eq:physical}.
\item \textbf{Compute sensitivity.} For each candidate, form $\phi$ by \eqref{eq:phi}, compute $b_i^{(j)}$ by \eqref{eq:mass-computable}, and assemble $J$ in \eqref{eq:Jacobian}.
\item \textbf{Reject ill-posed pairs.} If \eqref{eq:identifiability} fails or $\kappa_2(J)$ exceeds a prescribed threshold, replace one resonance by another branch or mode and repeat.
\item \textbf{Resolve ambiguity.} If two positive candidates remain, insert a third resonance $k^{(3)}$ into \eqref{eq:bilinear} or use a second mode $\ell'$; select the smallest residual.
\item \textbf{Report uncertainty.} Return $(\alpha_1,\alpha_2)$ with first-order estimate
\[
\|\delta\alpha\|_2
\lesssim
\|J^{-1}\|_2\,\varepsilon
\sqrt{(k^{(1)})^2+(k^{(2)})^2}.
\]
\end{enumerate}
\end{algorithm}

\section{Numerical experiments}\label{sec:numerics}
All computations use the shell $R_1=1$, $R_2=2$ with source at
$r_0=1.5$ in double precision. For a truncation order $N$ write
$G_N=\phi_d+v_N$ and define the boundary residual
\begin{equation}\label{eq:residual}
\rho_i(N)=\big\|\mathcal B_iG_N\big\|_{L^\infty(\{|x|=R_i\})},\qquad i=1,2.
\end{equation}

\subsection{Forward validation and spectral convergence}
Figure~\ref{fig:forward} shows $|G|$ in the two-dimensional annulus and on
the $xz$-section of the three-dimensional shell, together with the
convergence of the modal series against reference truncations
$N_{\rm ref}=150$ (2D) and $L_{\rm ref}=80$ (3D). The residuals
\eqref{eq:residual}, measured as the maximum of $|\mathcal B_iG_N|$ over
$400$ equispaced boundary points, remain at the $10^{-16}$ level for
every $N$ (each modal system \eqref{eq:modal-system} is solved exactly,
so both Robin conditions hold to machine precision; the residuals are
omitted from Figure~\ref{fig:forward}(c) since they sit at round-off).
The field error decays geometrically at the rate
$\max(R_1/r_0,r_0/R_2)^N=(3/4)^N$ predicted by
Remark~\ref{rem:convergence}; the initial plateau of the
three-dimensional curve in Figure~\ref{fig:forward}(c) is the
pre-asymptotic regime in which the decay sets in only once the
truncation exceeds the angular localisation scale $kr_0$ of the source.
Reciprocity $G(x,x_0)=G(x_0,x)$ was checked to
$10^{-14}$ on interior point pairs.

\begin{figure}[htbp]
\centering
\includegraphics[width=\textwidth]{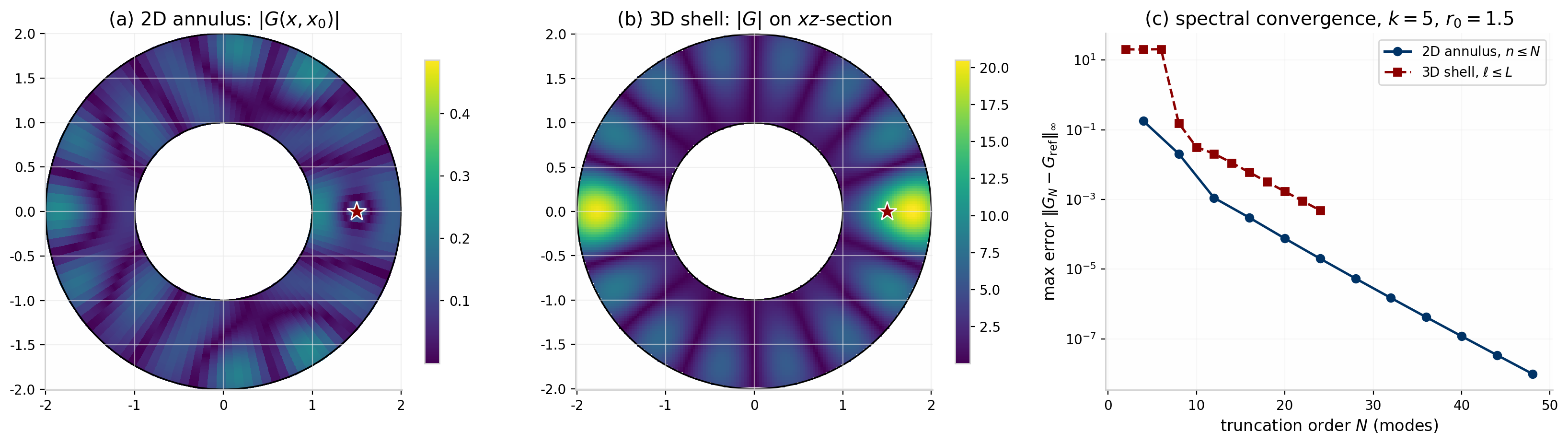}
\caption{Forward benchmark, $k=5$, $\alpha_1=1$, $\alpha_2=2$, $r_0=1.5$.
(a) Two-dimensional annulus $|G|$. (b) Three-dimensional shell $|G|$ on
the $xz$-section ($\star$: source). (c) Spectral convergence of the
modal series against $N_{\rm ref}=150$ (2D) and $L_{\rm ref}=80$ (3D);
the dashed slope corresponds to the rate $(3/4)^N$ of
Remark~\ref{rem:convergence}.}
\label{fig:forward}
\end{figure}

\subsection{Resonance branches, transparent interval, monotonicity}
Figure~\ref{fig:spectrum}(a) plots $\Delta_\ell$ for $\ell=0,1,2$ at
$\alpha_1=\alpha_2=0.5$; the first zeros
$k_{c,0}=0.9882$, $k_{c,1}=1.3586$, $k_{c,2}=1.8788$ confirm
$k_c=k_{c,0}$, and the shaded interval $(0,k_c)$ is resonance-free.
Panels (b)--(c) track the fundamental branch along the diagonal and the
fixed-impedance cross-sections: $k_{0,1}$ increases strictly from its
Neumann--Neumann value (the zero mode $k_{0,1}(0,0)=0$, cf.\ Remark~\ref{rem:zeromode}) to the Dirichlet--Dirichlet limit, in agreement
with Theorem~\ref{thm:basic}. Panel (d) compares the numerical threshold $k_c$ with the leading-order
term $15\eta/7$ of \eqref{eq:kc} and with the global rational fit of
Remark~\ref{rem:global-rational} over $\eta\in[3\times10^{-3},100]$
(horizontal axis: the common impedance $\eta=\alpha_1=\alpha_2$):
the squared threshold reproduces the leading-order estimate to within
$0.06\%$ at $\eta=3\times10^{-3}$ and to within $0.2\%$ for
$\eta\le10^{-2}$, in agreement with the second-order expansion
\eqref{eq:kc-2nd}; for larger $\eta$ the bare leading-order estimate
deviates as predicted (relative error $375\%$ at $\eta=100$), while the
rational fit agrees with the exact $k_c$ to within $0.03\%$ over the
whole range and approaches the Dirichlet--Dirichlet limit
$k_c^{DD}=\pi/\delta$.

\begin{figure}[htbp]
\centering
\includegraphics[width=0.7\textwidth]{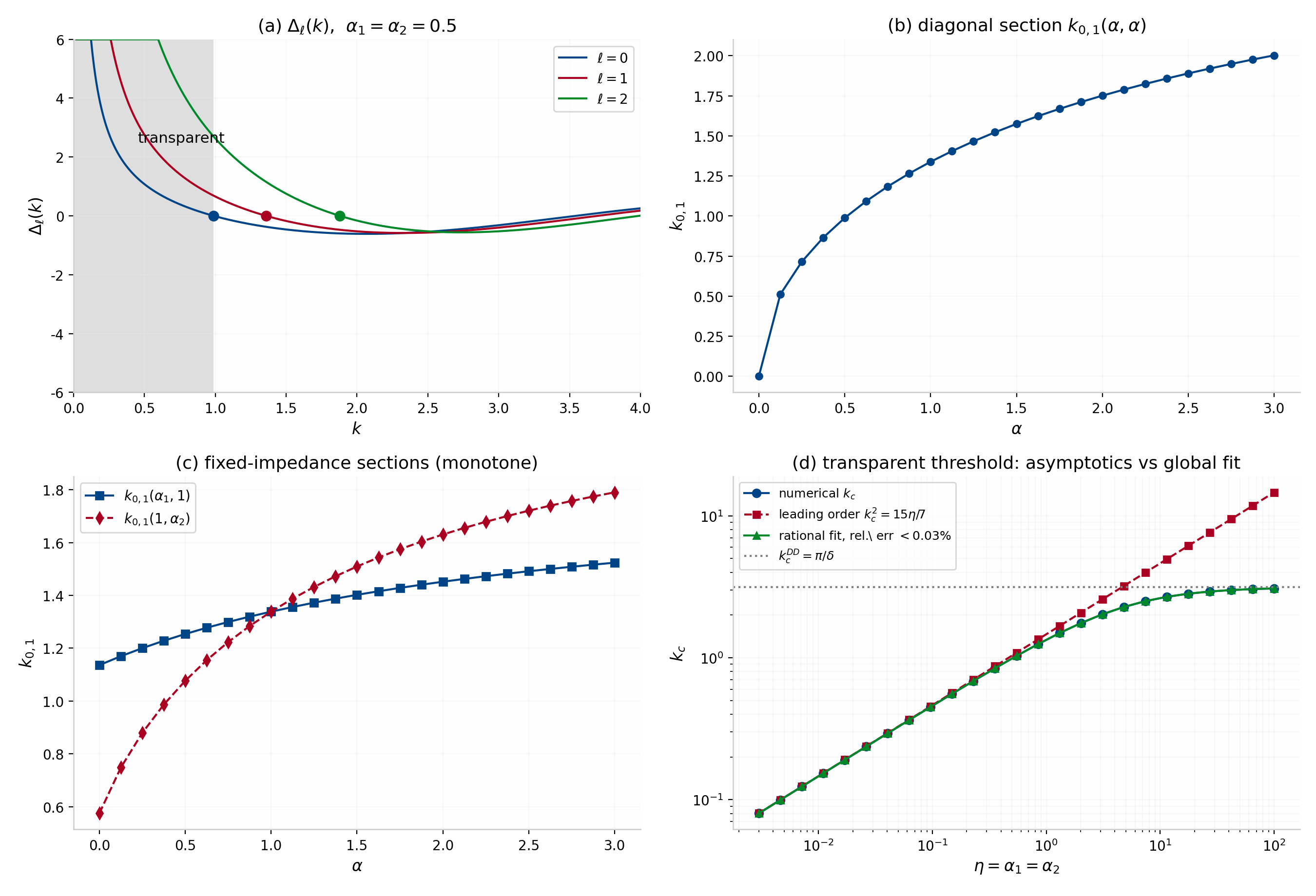}
\caption{Resonance map and transparent band. (a) $\Delta_\ell(k)$ for the
three-dimensional shell in the spherical-Bessel normalisation of
\eqref{eq:B1-3D}--\eqref{eq:B2-3D},
$\ell=0,1,2$, $\alpha_1=\alpha_2=0.5$ ($\eta=0.5$ in the notation of
Remark~\ref{rem:global-rational}); dots mark zeros, shaded band
$(0,k_c)$ is resonance-free. (b) Diagonal section $k_{0,1}(\alpha,\alpha)$.
(c) Fixed-impedance sections $k_{0,1}(\alpha_1,1)$ and $k_{0,1}(1,\alpha_2)$.
(d) Numerical $k_c$ versus the leading-order estimate $k_c^2=15\eta/7$
and the rational fit of Remark~\ref{rem:global-rational} over
$\eta\in[3\times10^{-3},100]$; the dotted line marks
$k_c^{DD}=\pi/\delta$.}
\label{fig:spectrum}
\end{figure}

\subsection{High-frequency spacing law}
Figure~\ref{fig:spacing} verifies Theorem~\ref{thm:spacing} for the
three-dimensional shell with $\alpha_1=1$, $\alpha_2=2$. The raw spacing
approaches the Weyl value $\pi/\delta=\pi$ from below (the first
spacings of the $\ell=0$ branch are $2.339, 2.818, 2.992, 3.059$), and
the rescaled residual $(\pi/\delta-s_p)\,k_p^2\delta^2/(\pi S_{\ell,3})$ converges
to $1$ with $S_{\ell,3}=\alpha_1+\alpha_2+(1+\ell(\ell+1)/2)\delta/(R_1R_2)$
as justified by Lemma~\ref{lem:Delta-asym}, which shows that the amplitude corrections of the large-argument Bessel expansions assemble into the curvature constant $K_{\ell,3}$.
The relative error of the two-term formula decays as $O(p^{-3})$ (Lemma~\ref{lem:odd} rules out the competing $O(p^{-2})$ alternating term),
with
log-log slope matching the reference line. The two-dimensional annulus
obeys the same law with $K_{n,2}=(4n^2+3)/8$: for $n=0,1,2$ and
$\alpha_1=1$, $\alpha_2=2$ the rescaled residual
$(\pi/\delta-s_p)k_p^2\delta^2/(\pi S_{n,2})$ starts at
$1.11,\,1.27,\,1.37$ respectively and decreases monotonically toward
$1$, in agreement with \eqref{eq:spacing}.

\begin{figure}[htbp]
\centering
\includegraphics[width=\textwidth]{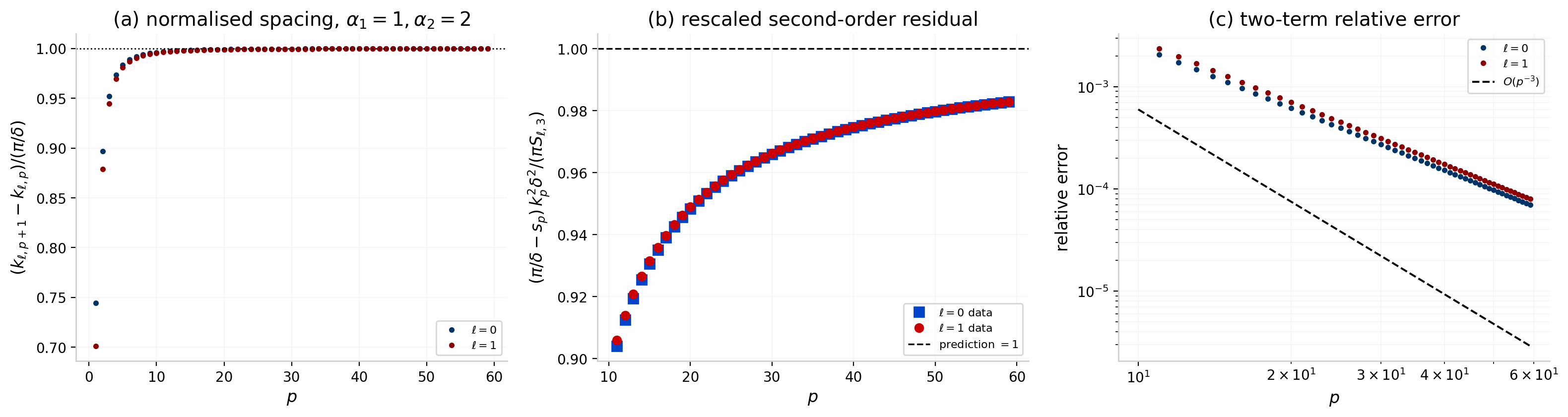}
\caption{High-frequency spacing, $\alpha_1=1$, $\alpha_2=2$.
(a) Normalised spacing $(k_{\ell,p+1}-k_{\ell,p})/(\pi/\delta)$.
(b) Rescaled second-order residual against prediction $=1$.
(c) Relative error of the two-term formula; reference slope $-3$.}
\label{fig:spacing}
\end{figure}

\subsection{Corner asymptotics}\label{sec:corners-num}
Table~\ref{tab:corners} and Figure~\ref{fig:corners} verify
Theorem~\ref{thm:corners} for the three-dimensional shell $R_1=1$,
$R_2=2$. The four limiting eigenvalues are computed as the first roots
of $\mathrm D_\ell$ (N--N), $\mathrm C_\ell$ (N--D), $\mathrm B_\ell$
(D--N) and $\mathrm A_\ell$ (D--D), cf.\ \eqref{eq:corner-eqs}; the
analytic first-order coefficients are evaluated from the boundary
masses \eqref{eq:bmass} and the squared normal derivatives of the
limiting eigenfunctions, and the mixed coefficient from
\eqref{eq:mixed-coeff}. The fitted coefficients are obtained by least
squares along approach rays---$\alpha_1=\alpha_2=\eta\downarrow0$
(N--N), $\alpha_2=1/\varepsilon$, $\alpha_1\downarrow0$ (N--D),
$\alpha_1=1/\varepsilon$, $\alpha_2\downarrow0$ (D--N), and
$\alpha_1=\alpha_2=1/\mu\uparrow\infty$ (D--D).

Two points are worth emphasising. First, the D--D mixed coefficient
cannot be read off from the diagonal ray alone: since the remainder in
Theorem~\ref{thm:corners}(iv) is $O(\mu_1^2+\mu_2^2)$, the diagonal
quadratic coefficient equals $g_{\ell,p}+a_1+a_2$ with unknown pure
quadratic terms $a_i\mu_i^2$ (numerically $\approx +13.9$ for $\ell=1$ here, with positive sign: the quadratic term partly counteracts the negative linear Dirichlet corrections $f_1\mu_1+f_2\mu_2$); $g_{\ell,p}$ itself is extracted from the mixed difference
quotient $k(\mu,\mu)-k(\mu,0)-k(0,\mu)+k(0,0)=g_{\ell,p}\mu^2+O(\mu^3)$,
evaluated at $\mu=10^{-2}$ and $5\times10^{-3}$ and extrapolated
linearly in $\mu$. Second, the zero mode at the N--N corner
(Remark~\ref{rem:zeromode}) is visible in the data: for $\ell=0$ the
fundamental branch emanates from $k=0$ with the square-root law
\eqref{eq:kc}, and the linear N--N asymptotics are verified on the
$p=2$ branch instead (last row of Table~\ref{tab:corners}).

The agreement is uniform. The fitted intercepts of the N--D and D--N
rays reproduce the limiting eigenvalues to four decimal places or
better (e.g.\
$2.304878$ versus $k_{1,1}^{ND}=2.304884$; $1.435612$ versus
$k_{1,1}^{DN}=1.435635$); the fitted first-order coefficients agree
with the boundary-mass values to within $2.7\%$, consistent with the $O(\eta)$, $O(\varepsilon)$ and $O(\mu)$ fit
bias; and the fitted mixed coefficients reproduce \eqref{eq:mixed-coeff}
to within $0.01\%$, confirming the values $g_{\ell,1}\approx
6.28,\,6.20,\,5.96$ quoted in Theorem~\ref{thm:corners}(iv), with
$g_{0,1}=6.28319$ agreeing with $2\pi$ to the displayed accuracy. This is not a numerical coincidence: Theorem~\ref{thm:g0} gives the closed form $g_{0,1}=2\pi/(R_2-R_1)^3$ for every shell, so that in particular $g_{0,1}=2\pi$ for the benchmark shell $R_1=1$, $R_2=2$.

\begin{figure}[htbp]
\centering
\includegraphics[width=\textwidth]{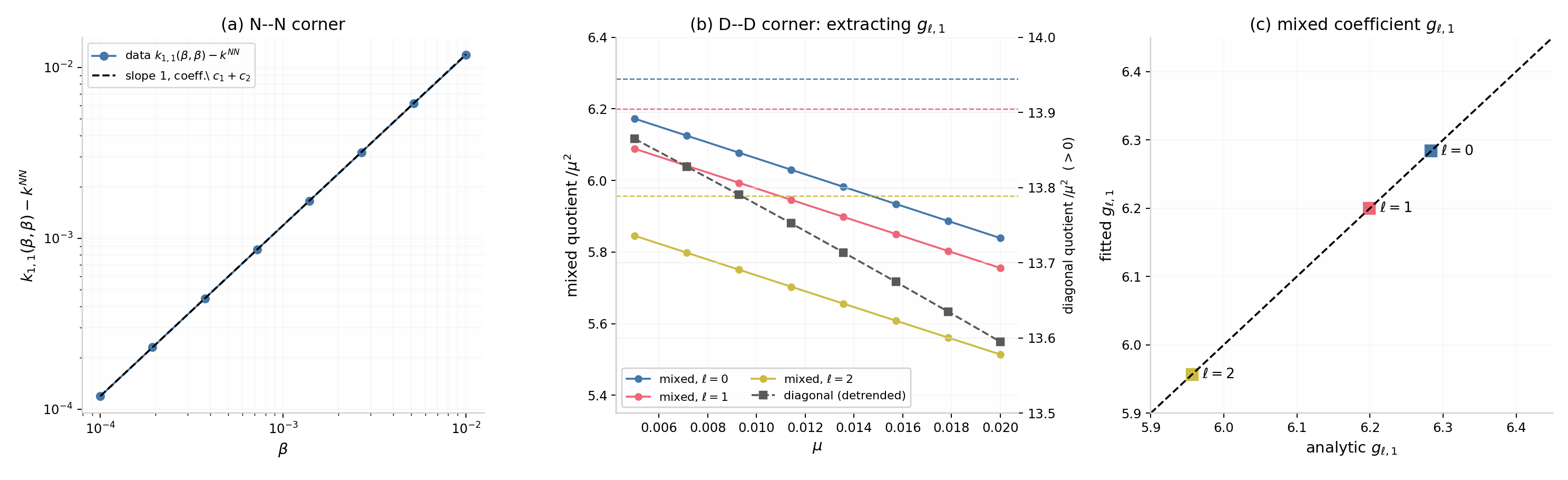}
\caption{Corner asymptotics for the three-dimensional shell $R_1=1$, $R_2=2$ (Table~\ref{tab:corners}). (a) N--N approach: $k_{1,1}(\eta,\eta)-k_{1,1}^{NN}$ versus $\eta$ on log--log scales; the dashed line has slope one with coefficient
$c_1+c_2$. (b) D--D approach: mixed difference quotient $[k(\mu,\mu)-k(\mu,0)-k(0,\mu)+k(0,0)]/\mu^{2}$ for $\ell=0,1,2$ (markers) converging, as $\mu\downarrow0$, to the analytic values \eqref{eq:mixed-coeff} (dashed horizontal lines); for comparison, the detrended diagonal quotient $[k(\mu,\mu)-k^{DD}-(f_1+f_2)\mu]/\mu^{2}$ for $\ell=1$ (grey squares, right axis, which carries the computed positive values $13.6$--$13.9$) estimates $g_{1,1}+a_1+a_2\approx+13.9$. (c) Analytic versus fitted mixed coefficients of Table~\ref{tab:corners} (identity line).}
\label{fig:corners}
\end{figure}

\begin{table}[htbp]
\centering
\renewcommand{\arraystretch}{1.25}
\caption{Corner asymptotics of Theorem~\ref{thm:corners} for the
three-dimensional shell $R_1=1$, $R_2=2$: analytic values (first-order
coefficients from the boundary masses \eqref{eq:bmass} and squared
normal derivatives of the limiting eigenfunctions; mixed coefficient
from \eqref{eq:mixed-coeff}) versus least-squares fits along approach
rays. The deviation of the fitted first-order coefficients (at most
$2.7\%$) is consistent with the $O(\eta)$, $O(\varepsilon)$,
$O(\mu)$ bias of the fits. All rows use the fundamental branch except
the last one ($p=2$).}
\label{tab:corners}\vspace{0.3cm}
\centering\setlength{\tabcolsep}{3pt}
\begin{tabular}{c|l|c|c|c}
\hline
$(\ell,p)$ & quantity & analytic & fitted & rel.\ dev.\\
\hline
 & $k^{NN}$; $c_1$; $c_2$ & $0.920134$; $0.20286$; $0.98971$ & $0.920135$; $0.20319$; $0.98129$ & ---; $0.2\%$; $0.9\%$\\
 & $k^{ND}$; $d_1$; $d_2$ & $2.304884$; $0.27860$; $-1.61387$ & $2.304878$; $0.27500$; $-1.60671$ & ---; $1.3\%$; $0.4\%$\\
$(1,1)$ & $k^{DN}$; $e_1$; $e_2$ & $1.435635$; $-1.27642$; $0.86867$ & $1.435612$; $-1.24760$; $0.84705$ & ---; $2.3\%$; $2.5\%$\\
 & $g_{1,1}$ & $6.19929$ & $6.19971$ & $0.007\%$\\
\hline
 & $k^{NN}$; $c_1$; $c_2$ & $1.575588$; $0.08974$; $0.64120$ & $1.575589$; $0.09007$; $0.63825$ & ---; $0.4\%$; $0.5\%$\\
 & $k^{ND}$; $d_1$; $d_2$ & $2.769322$; $0.21017$; $-1.56006$ & $2.769318$; $0.20765$; $-1.55428$ & ---; $1.2\%$; $0.4\%$\\
$(2,1)$ & $k^{DN}$; $e_1$; $e_2$ & $1.858787$; $-0.85250$; $0.69339$ & $1.858768$; $-0.82914$; $0.68014$ & ---; $2.7\%$; $1.9\%$\\
 & $g_{2,1}$ & $5.95656$ & $5.95661$ & $0.001\%$\\
\hline
$(0,1)$ & $k^{DD}$; $g_{0,1}$ & $3.141593$; $6.28319$ & $3.141593$; $6.28380$ & ---; $0.01\%$\\
\hline
$(0,2)$ & $k^{NN}$; $c_1+c_2$ & $6.360678$; $0.30605$ & $6.360679$; $0.30589$ & ---; $0.05\%$\\
\hline
\end{tabular}
\end{table}

\subsection{Exact and noisy impedance recovery}
Figure~\ref{fig:recovery} and Table~\ref{tab:recovery} test
Proposition~\ref{prop:recovery} on a $15\times15$ grid over
$(\alpha_1,\alpha_2)\in[0.2,5]^2$, using the first three resonances of
the mode $\ell=1$ (3D) and $n=1$ (2D): the first two define the quadratic
\eqref{eq:quadratic} and the third selects the physical root
(cf.\ Remark~\ref{rem:two-candidates}). Noiseless recovery
is exact to solver precision (median $3.0\times10^{-15}$ in 3D and
$4.5\times10^{-15}$ in 2D). Under relative frequency noise
$k^{(j)}\!\to\!k^{(j)}(1+\varepsilon\eta_j)$, $\eta_j\sim\mathcal N(0,1)$,
the error grows linearly in $\varepsilon$, in agreement with
\eqref{eq:error-bound} (the $\varepsilon=10^{-2}$ columns are a stress
test: the errors are then comparable to the impedances themselves, so
the first-order bound no longer applies quantitatively); the $p95$ error
and failure rate at $\varepsilon=10^{-2}$ are larger in 2D, whose
condition number reaches $1.8\times10^3$ (versus $3.2\times10^2$ in 3D).
In both cases $96$--$97\%$ of the impedance plane is well conditioned
($\kappa_2(J)<100$), the ill-conditioned cells concentrating in the
near-diagonal strip $\alpha_1\approx\alpha_2$, exactly where
$b_1\approx b_2$ and the two columns of \eqref{eq:Jacobian} become
nearly parallel (Proposition~\ref{prop:conditioning}; see also
Theorem~\ref{thm:corners}). A second admissible root occurred at most
once per grid point and was always removed by the third resonance.

\begin{figure}[htbp]
\centering
\includegraphics[width=0.95\textwidth]{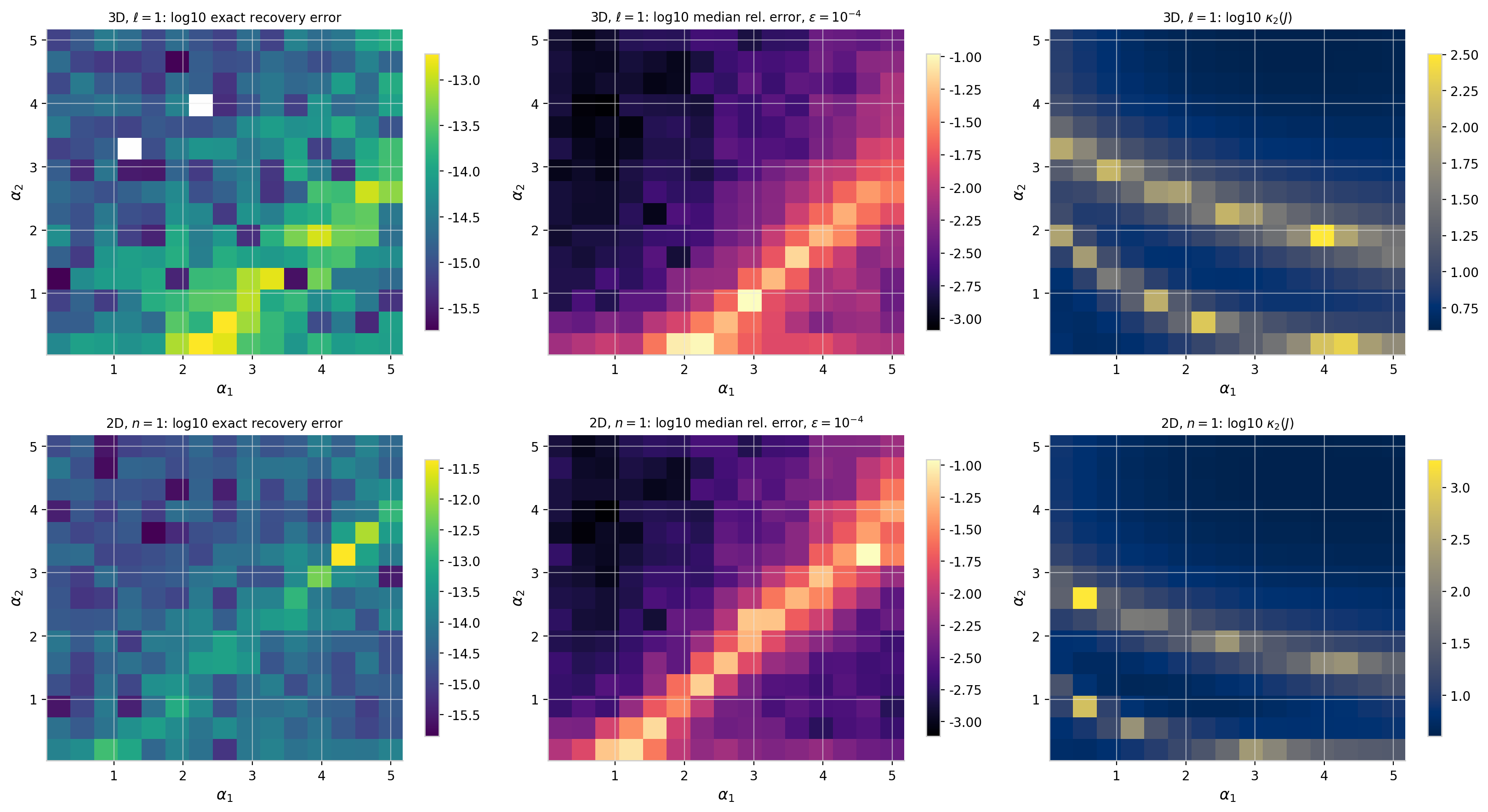}
\caption{Resonance-based impedance recovery. Rows: 3D shell, mode
$\ell=1$ (top); 2D annulus, mode $n=1$ (bottom). Columns: exact
recovery error (left); median relative error at $\varepsilon=10^{-4}$
over 30 realisations (middle); $\kappa_2(J)$ from \eqref{eq:Jacobian}
(right).}
\label{fig:recovery}
\end{figure}
\vspace{-0.3cm}

\begin{table}[htbp]
\centering
\renewcommand{\arraystretch}{1.3} 
\caption{Recovery statistics on the $15\times15$ grid over
$[0.2,5]^2$ using the first two resonances for the quadratic and a
third for disambiguation; noise model $k^{(j)}\!\to\!k^{(j)}
(1+\varepsilon\eta_j)$, 30 realisations per point. The last column is
the fraction of grid points at which the relative recovery error in
$(\alpha_1,\alpha_2)$ exceeded $0.5$ in at least one realisation.}
\label{tab:recovery}\vspace{0.3cm}
\begin{tabular}{c|c|c|c|c}
\hline
case & noise $\varepsilon$ & median rel.\ err. & p95 rel.\ err. &
frac.\ err.$>0.5$\\
\hline
3D, $\ell=1$ & 0        & $3.0\times10^{-15}$ & $4.7\times10^{-14}$ & --\\
3D, $\ell=1$ & $10^{-4}$& $3.0\times10^{-3}$  & $3.5\times10^{-2}$  & 0.024\\
3D, $\ell=1$ & $10^{-2}$& $3.0\times10^{-1}$  & $2.5\times10^{0}$   & 0.36\\
\hline
2D, $n=1$    & 0        & $4.5\times10^{-15}$ & $5.1\times10^{-14}$ & --\\
2D, $n=1$    & $10^{-4}$& $3.0\times10^{-3}$  & $1.4\times10^{-1}$  & 0.043\\
2D, $n=1$    & $10^{-2}$& $4.0\times10^{-1}$  & $4.3\times10^{0}$   & 0.44\\
\hline
\end{tabular}
\end{table}

\begin{remark}[Reflection pairs are high-frequency blind]
The reflection map of Remark~\ref{rem:reflection} preserves the
impedance sum $\alpha_1+\alpha_2$, and the spacing correction
\eqref{eq:S} depends on $(\alpha_1,\alpha_2)$ only through this sum.
Numerically, the $\ell=1$ spectra of the reflection pair $(2,3)$ and
$(1.5,3.5)$ differ by $7.7\times10^{-3}$ in the first resonance but by
less than $5\times10^{-4}$ in the second and third. Consequently
high-frequency resonances of one non-radial mode resolve
reflection-related ambiguity poorly, and recovery should rely on the low-frequency spectrum, where the two admissible pairs are maximally separated. The decay of the discrepancy with the resonance index is consistent with the leading quantisation condition \eqref{eq:Delta-large-k}, in which the impedances enter only through the reflection-invariant sum $S_{\ell,d}=\alpha_1+\alpha_2+K_{\ell,d}\delta/\rho$; reflection-sensitive corrections appear only in the higher-order terms of the large-$k$ expansion.
\end{remark}

\begin{remark}[Practical branch assignment]\label{rem:assignment}
Branch assignment is most reliable at the \emph{lowest} resonances, for
two quantitative reasons. First, the transparent threshold $k_c$ is
computable from \eqref{eq:kc} a priori, and no resonance of any angular
order lies below it, so the fundamental resonance of each mode is
separated from the spectral edge by a controlled margin. Second,
reflection-related ambiguity is concentrated at low frequency (cf. the
preceding remark: the first resonances of the pair $(2,3)$ and its
reflection partner $(1.5,3.5)$ differ by $7.7\times10^{-3}$, the higher
ones by less than $5\times10^{-4}$). Since all branches share the
same asymptotic spacing $\pi/\delta$, high-frequency resonances of
different angular orders interlace ever more tightly and cannot be
assigned reliably from spacing data alone. We therefore recommend
assigning branches at the bottom of the spectrum (using the $k_c$
threshold and the monotonicity of $k_{c,\ell}$ in $\ell$), and using the
higher resonances only as consistency checks; this is also what
Procedure~\ref{alg:recovery} assumes in its first step.
\end{remark}

\section{Conclusion}\label{sec:conclusion}
This paper develops an explicit theory, computable to machine precision, of the
two-impedance Robin Helmholtz problem on annular and spherical-shell
domains. The forward problem is solved in closed form in every dimension
$d\ge2$: Graf's addition theorem and the hyperspherical addition theorem
reduce each angular mode to an explicit $2\times2$ linear system, yielding
series kernels that satisfy both Robin conditions to machine precision and
converge geometrically---at the explicit rate of
Remark~\ref{rem:convergence}---uniformly on the closed shell. Because
the modal coefficients are rational functions of the two impedances, the
dependence of the wave field on the boundary parameters is exposed term
by term, and the classical single-boundary ball formulas are recovered as
degenerate limits.

The same explicitness determines the resonance spectrum quantitatively.
The branches are positive, simple, real-analytic and strictly increasing
in both impedances; their asymptotic scaling at the four corners of the
impedance plane is resolved to first order, with coefficients given by
the boundary masses and normal derivatives of the limiting
eigenfunctions, and the mixed second-order coefficient at the
Dirichlet--Dirichlet corner is given explicitly by \eqref{eq:mixed-coeff}
(nonzero in general). At low frequencies the determinant admits an
explicit two-term expansion whose balance yields a rigorous second-order
small-impedance expansion of the threshold of the transparent interval,
extended by a global rational approximation accurate over the entire
impedance range. At high frequencies a complete large-wavenumber
expansion produces the universal spacing $\pi/(R_2-R_1)$ together with a
second-order curvature correction coupling the impedance sum to the
shell geometry through $K_{\ell,d}$; the $O(p^{-3})$ error term is
secured by the parity argument of Lemma~\ref{lem:odd}.

The bilinearity of the characteristic determinant then turns
resonance-based impedance recovery into an algebraic problem: two
resonant frequencies of one non-radial angular mode generate at most two
candidate impedance pairs through an explicit quadratic equation, with
the physical pair selected by a third resonance, an exact
reflection-symmetry obstruction delimiting the region in which the
radial spherical mode is unidentifiable, and
a Jacobian criterion, based on the Hellmann--Feynman boundary masses,
that selects well-conditioned resonance pairs and predicts the noise
amplification factor; local uniqueness is characterised exactly by this
criterion. Numerical
benchmarks confirm every regime: machine-precision boundary residuals
and geometric convergence for the forward kernel, agreement of the
transparent threshold with the rational fit to within $0.03\%$ over
$\eta\in[3\times10^{-3},100]$, the predicted $O(p^{-3})$ decay in the
spacing law, corner asymptotics matching the boundary-mass coefficients
to within the fit bias with the mixed coefficient reproduced to $0.01\%$,
and exact recovery to solver precision with error growing
linearly in the noise level and concentrated, as predicted, in the
near-diagonal ill-conditioned strip.

Natural extensions include complex impedances with $\mathrm{Re}(\alpha_i)>0$ and
the associated non-self-adjoint resonance picture, multilayer or
eccentric shells, time-domain transforms of the explicit kernels, and
stochastic recovery under uncertain geometry, where the conditioning
criteria of Section~\ref{sec:sensitivity} provide a quantitative
starting point.

\section*{Statements and Data Availability}

The author declares that he has no conflict of interest.
All numerical results reported in this paper were generated by the
author; source code reproducing all figures and tables is available
from the author upon reasonable request.

\section*{Acknowledgements}
This work is supported by the Jiangsu Provincial Scientific Research Center of Applied Mathematics under Grant No. BK20233002.

\appendix

\section{Two-dimensional annulus: complete working formulas}\label{app:2D}
For $d=2$, $\nu=0$ and it is simplest to use the Fourier basis $e^{in\theta}$, $n\in\mathbb Z$. The free solution is
\[
\phi_2(x-x_0)=\frac{i}{4}H_0^{(1)}(k|x-x_0|),
\]
and Graf's addition theorem gives
\[
\phi_2(x-x_0)=
\begin{cases}
\displaystyle \frac{i}{4}\sum_{n\in\mathbb Z} H_n^{(1)}(kr_0)J_n(kr)e^{in(\theta-\theta_0)},& r<r_0,\\[1em]
\displaystyle \frac{i}{4}\sum_{n\in\mathbb Z} J_n(kr_0)H_n^{(1)}(kr)e^{in(\theta-\theta_0)},& r>r_0.
\end{cases}
\]
For $Z_n\in\{J_n,Y_n,H_n^{(1)}\}$ define
\begin{align}
\mathcal B_1[Z_n]
&=
\left(\alpha_1-\frac{n}{R_1}\right)Z_n(kR_1)+kZ_{n+1}(kR_1),\label{eq:B1-2D}\\
\mathcal B_2[Z_n]
&=
\left(\alpha_2+\frac{n}{R_2}\right)Z_n(kR_2)-kZ_{n+1}(kR_2).\label{eq:B2-2D}
\end{align}
The correction is
\[
v(r,\theta)=
\sum_{n\in\mathbb Z}
\left(A_nJ_n(kr)+B_nY_n(kr)\right)e^{in(\theta-\theta_0)} .
\]
Matching Fourier coefficients gives
\[
\begin{pmatrix}
\mathcal B_1[J_n] & \mathcal B_1[Y_n]\\
\mathcal B_2[J_n] & \mathcal B_2[Y_n]
\end{pmatrix}
\begin{pmatrix}A_n\\ B_n\end{pmatrix}
=
\begin{pmatrix}
-\dfrac{i}{4}H_n^{(1)}(kr_0)\mathcal B_1[J_n]\\[6pt]
-\dfrac{i}{4}J_n(kr_0)\mathcal B_2[H_n^{(1)}]
\end{pmatrix}.
\]
Hence
\begin{align}
A_n&=
\frac{i}{4\Delta_n}
\left(
J_n(kr_0)\mathcal B_2[H_n^{(1)}]\,\mathcal B_1[Y_n]
-
H_n^{(1)}(kr_0)\mathcal B_1[J_n]\,\mathcal B_2[Y_n]
\right),\label{eq:An-2D}\\
B_n&=
\frac{i}{4\Delta_n}
\left(
H_n^{(1)}(kr_0)\mathcal B_1[J_n]\,\mathcal B_2[J_n]
-
J_n(kr_0)\mathcal B_1[J_n]\,\mathcal B_2[H_n^{(1)}]
\right),\label{eq:Bn-2D}
\end{align}
with
\begin{equation}\label{eq:Delta-n-2D}
\Delta_n=
\mathcal B_1[J_n]\,\mathcal B_2[Y_n]
-
\mathcal B_2[J_n]\,\mathcal B_1[Y_n].
\end{equation}
Therefore the two-dimensional Robin Green's function is
\begin{equation}\label{eq:G-2D}
G(x,x_0)=
\frac{i}{4}H_0^{(1)}(k|x-x_0|)
+
\sum_{n\in\mathbb Z}
\left(A_nJ_n(kr)+B_nY_n(kr)\right)e^{in(\theta-\theta_0)}
\end{equation}
with $A_n,B_n$ given by \eqref{eq:An-2D}--\eqref{eq:Bn-2D}. If one prefers to sum over $n\ge0$, use
$J_{-n}=(-1)^nJ_n$, $Y_{-n}=(-1)^nY_n$, and $H_{-n}^{(1)}=(-1)^nH_n^{(1)}$, keeping the $n=0$ term undoubled.

\section{Proofs of the spectral lemmas}\label{app:proofs}

\subsection{Proof of Lemma~\ref{lem:Delta-asym}}\label{app:Delta-asym}

Write $\theta_\ell(z)=z-\ell\pi/2$ and $K_\ell=1+\frac{\ell(\ell+1)}{2}$.
From \cite[\S10.17]{Olver2010},
\begin{align*}
&j_\ell(z)=\frac{\sin\theta_\ell(z)}{z}
+\frac{\ell(\ell+1)}{2}\frac{\cos\theta_\ell(z)}{z^2}+O(z^{-3}),\\
&y_\ell(z)=-\frac{\cos\theta_\ell(z)}{z}
+\frac{\ell(\ell+1)}{2}\frac{\sin\theta_\ell(z)}{z^2}+O(z^{-3}),\\
&j_{\ell+1}(z)=-\frac{\cos\theta_\ell(z)}{z}
+\frac{(\ell+1)(\ell+2)}{2}\frac{\sin\theta_\ell(z)}{z^2}+O(z^{-3}),\\
&y_{\ell+1}(z)=-\frac{\sin\theta_\ell(z)}{z}
-\frac{(\ell+1)(\ell+2)}{2}\frac{\cos\theta_\ell(z)}{z^2}+O(z^{-3}).
\end{align*}
Substituting these into \eqref{eq:B1-3D}--\eqref{eq:B2-3D} gives the
four trace expansions ($\theta_i=\theta_\ell(kR_i)$, $z_i=kR_i$)
\begin{align*}
\mathcal B_1[j_\ell]
&=-\frac{\cos\theta_1}{R_1}
+\frac{(\alpha_1R_1+K_\ell)\sin\theta_1}{kR_1^2}+O(k^{-2}),
&
\mathcal B_1[y_\ell]
&=-\frac{\sin\theta_1}{R_1}
-\frac{(\alpha_1R_1+K_\ell)\cos\theta_1}{kR_1^2}+O(k^{-2}),\\
\mathcal B_2[j_\ell]
&=\frac{\cos\theta_2}{R_2}
+\frac{(\alpha_2R_2-K_\ell)\sin\theta_2}{kR_2^2}+O(k^{-2}),
&
\mathcal B_2[y_\ell]
&=\frac{\sin\theta_2}{R_2}
+\frac{(K_\ell-\alpha_2R_2)\cos\theta_2}{kR_2^2}+O(k^{-2}).
\end{align*}
In particular the traces are $O(1)$, each carrying an explicit
$O(k^{-1})$ correction. Since $\theta_2-\theta_1=k\delta$, the leading
terms give
\[
-\frac{\cos\theta_1}{R_1}\frac{\sin\theta_2}{R_2}
+\frac{\cos\theta_2}{R_2}\frac{\sin\theta_1}{R_1}
=
-\frac{\sin(\theta_2-\theta_1)}{R_1R_2}
=
-\frac{\sin(k\delta)}{R_1R_2}.
\]
At order $k^{-1}$, both products contribute the \emph{same} coefficient,
\[
\frac{(\alpha_1+\alpha_2)\rho+K_\ell\delta}{R_1^2R_2^2},
\]
once to $\cos\theta_1\cos\theta_2$ and once to
$\sin\theta_1\sin\theta_2$; hence
\begin{align*}
\Delta_\ell
&=
\frac{(\alpha_1+\alpha_2)\rho+K_\ell\delta}{kR_1^2R_2^2}
\bigl(\cos\theta_1\cos\theta_2+\sin\theta_1\sin\theta_2\bigr)
+O(k^{-2})\\
&=
\frac{1}{R_1R_2}\left[
-\sin(k\delta)
+\frac{(\alpha_1+\alpha_2)+K_\ell\,\delta/\rho}{k}\cos(k\delta)
+O\big((kR)^{-2}\big)\right],
\end{align*}
because $\cos\theta_1\cos\theta_2+\sin\theta_1\sin\theta_2=
\cos(\theta_2-\theta_1)=\cos(k\delta)$: the $\cos(k(R_1+R_2))$ terms
cancel identically. This proves \eqref{eq:Delta-asym-d3}. The
general-$d$ statement follows from the identical computation with
$J_\mu,Y_\mu$ in place of $j_\ell,y_\ell$, for which the centrifugal
shift $\mu=\ell+\nu$ produces
$K_{\ell,d}=((2\ell+d-2)^2+4d-5)/8$; the common factor
$C(k;R_1,R_2)$ is exactly the one recorded in the lemma.

\subsection{Proof of Lemma~\ref{lem:odd}: no even inverse powers}\label{app:odd}

Introduce Pr\"ufer variables $w=\rho\sin\theta$, $w'=k\rho\cos\theta$ with $\theta$ continuous. Then \eqref{eq:normal-form} gives the exact phase equation
\begin{equation}\label{eq:prufer}
\theta'=k-\frac{Q(r)}{k}\sin^2\theta ,
\end{equation}
and the boundary conditions read $\cot\theta(R_1)=\gamma_1/k$, $\cot\theta(R_2)=-\gamma_2/k$. Since the $p$-th eigenfunction has exactly $p-1$ interior zeros,
\[
\theta(R_2)-\theta(R_1)=p\pi+\arctan\frac{\gamma_1}{k}+\arctan\frac{\gamma_2}{k},
\]
and integrating \eqref{eq:prufer} yields the exact quantization identity
\begin{equation}\label{eq:quant}
k\delta=p\pi+\arctan\frac{\gamma_1}{k}+\arctan\frac{\gamma_2}{k}
+\frac1k\int_{R_1}^{R_2}Q(r)\sin^2\theta(r;k)\,dr .
\end{equation}
Write $\theta(r;k)=k(r-R_1)+\psi(r;k)$; then $|\psi|\le k^{-1}\|Q\|_{L^{1}}$ by \eqref{eq:prufer}, and using $\sin^2a-\sin^2b=\sin(a+b)\sin(a-b)$,
\[
\int_{R_1}^{R_2}Q\sin^2\theta\,dr
=
\int_{R_1}^{R_2}Q\sin^2(k(r-R_1))\,dr
+
\int_{R_1}^{R_2}Q\sin(2k(r-R_1))\psi\,dr
+
O(k^{-2}).
\]
For the first integral,
$\int_{R_1}^{R_2}Q\sin^2(k(r-R_1))\,dr=\frac12\int_{R_1}^{R_2}Q\,dr-\frac12\int_{R_1}^{R_2}Q\cos(2k(r-R_1))dr$,
and integration by parts (using $Q\in C^{2}$) gives
\[
\int_{R_1}^{R_2}Q\cos(2k(r-R_1))\,dr
=
\frac{Q(R_2)\sin(2k\delta)}{2k}+O(k^{-2}),
\]
where the $O(k^{-2})$ remainder is a trigonometric polynomial in $k\delta$ with coefficients depending on $Q,Q',Q''$. For the second integral, iterate \eqref{eq:prufer} once:
\[
\psi(r;k)=-\frac{1}{2k}\int_{R_1}^{r}Q(t)\,dt+O(k^{-2})
\]
uniformly in $r$; hence
\[
\int_{R_1}^{R_2}Q\sin(2k(r-R_1))\psi\,dr
=
-\frac{1}{2k}\int_{R_1}^{R_2}Q(r)\Bigl(\int_{R_1}^{r}Q(t)\,dt\Bigr)\sin(2k(r-R_1))\,dr+O(k^{-2})
=
O(k^{-2}),
\]
the leading term being oscillatory and therefore $O(k^{-1})$ by a further integration by parts. Altogether,
\[
\int_{R_1}^{R_2}Q\sin^2\theta\,dr=\frac12\int_{R_1}^{R_2}Q\,dr-\frac{Q(R_2)\sin(2k\delta)}{4k}+J(k),
\qquad J(k)=O(k^{-2}),
\]
where $J$ may contain non-oscillatory terms, which after multiplication by $1/k$ enter the quantization equation only at $O(k^{-3})$. Substituting into \eqref{eq:quant} and using $\arctan x=x+O(x^{3})$,
\[
k\delta=p\pi+\frac{A}{k}-\frac{Q(R_2)\sin(2k\delta)}{4k^{2}}+O(k^{-3}),\qquad A=\gamma_1+\gamma_2+\frac12\int_{R_1}^{R_2}Q\,dr .
\]
At a root $k=k_{p}$, the first-order determinant asymptotics \eqref{eq:Delta-large-k} give $\tan(k_{p}\delta)=S_{\ell,d}/k_{p}+O(k_{p}^{-2})$, hence $\sin(k_{p}\delta)=O(p^{-1})$ and $\sin(2k_{p}\delta)=O(p^{-1})$; the term $-Q(R_2)\sin(2k_{p}\delta)/(4k_{p}^{2})$ is therefore $O(k_{p}^{-3})$ and is absorbed into the remainder, which proves \eqref{eq:odd} with a bounded sequence $\{B_p\}$.

\end{document}